\documentclass[11pt,reqno]{amsart}
\usepackage{amsthm,amssymb,amsmath}
\usepackage{xcolor}
\newtheorem{theorem}{Theorem}[section]
\newtheorem*{theorem*}{Theorem}
\newtheorem{lemma}[theorem]{Lemma}

\newtheorem{proposition}{Proposition}[section]

\theoremstyle{definition}

\newtheorem{definition}{\sc Definition}[section]
\newtheorem*{definition*}{\sc Definition}

\newtheorem{example}{\bf Example}[section]
\newtheorem{remark}{Remark}[section]
\newtheorem*{remark*}{Remark}

\newtheorem*{example*}{\sc Example}
\newtheorem*{examples}{\sc Examples}

\newcommand{\loc}{{\rm loc}}

\newcommand{\Real}{{\rm Re}}

\newcommand{\mydiv}{{\rm div\,}}

\expandafter\def\expandafter\normalsize\expandafter{%
    \normalsize
    \setlength\abovedisplayshortskip{8pt}
    \setlength\belowdisplayshortskip{8pt}
}

\begin{document}

\title[Parabolic equations with singular vector fields]{Regularity for parabolic equations with singular non-zero divergence vector fields}

\author{D.\,Kinzebulatov}

\address{Universit\'{e} Laval, D\'{e}partement de math\'{e}matiques et de statistique, 1045 av.\,de la M\'{e}decine, Qu\'{e}bec, QC, G1V 0A6, Canada}

\email{damir.kinzebulatov@mat.ulaval.ca}

\author{Yu.\,A.\,Sem\"{e}nov}

\address{University of Toronto, Department of Mathematics, 40 St.\,George Str, Toronto, ON, M5S 2E4, Canada}

\email{semenov.yu.a@gmail.com}



\thanks{The research of D.K.\,is supported by the Natural Sciences and Engineering Research Council of Canada (grant RGPIN-2017-05567)}

\begin{abstract}
We establish two-sided Gaussian bounds on the heat kernel of divergence-form parabolic equation with singular time-inhomogeneous vector field satisfying some minimal assumptions.
\end{abstract}

\maketitle


\section{Introduction}

In this paper we study regularity properties of solutions of parabolic equation
\begin{equation}
\label{eq1}
(\partial_t - \nabla \cdot a \cdot \nabla + b \cdot \nabla)u=0, \quad \text{ on } \mathbb R_+ \times \mathbb R^d, \quad d \geq 3,
\end{equation}
under broad assumptions on a singular (that is, locally unbounded) time-inhomogeneous vector field $b:\mathbb R_+ \times \mathbb R^d \rightarrow \mathbb R^d$. Here and everywhere below,
$\mathbb R_+:=[0,\infty[$,
$$
\nabla \cdot a \cdot \nabla u(x):=\sum_{i,j=1}^d \nabla_i \bigl(a_{ij} \nabla_j u(x)\bigr), \quad b \cdot \nabla u(x):=\sum_{k=1}^d b_k \nabla_k u(x).
$$
The matrix $a:\mathbb R_+ \times \mathbb R^d \rightarrow \mathbb R^d \otimes \mathbb R^d$ is symmetric, uniformly elliptic, i.e.\,
\begin{equation}
\label{a_hyp}
\tag{$H_{\sigma,\xi}$}
\begin{array}{c}
\text{$\sigma I \leq a(t,x) \leq \xi I$ \quad for some $\xi,\sigma>0$ for a.e.\,$(t,x)$},
\end{array}
\end{equation}
and, other than that, is assumed to be only measurable. 

\medskip

In the absence of drift term $b \cdot \nabla$, already the hypothesis $a \in H_{\sigma,\xi}$ provides local H\"{o}lder continuity of solutions to \eqref{eq1}. This is known since the works of E.\,De Giorgi \cite{DG} and J.\,Nash \cite{N}. Moreover, the corresponding heat kernel satisfies two-sided Gaussian bounds,  see Aronson \cite{Ar}, or see \cite{N} and Fabes-Stroock \cite{FS}. 

\medskip

However, with non-zero $b \cdot \nabla$, the regularity theory of \eqref{eq1} is not yet complete even in the case $a=I$. Some particularly important classes of vector fields $b$ are the divergence-free vector fields (i.e.\,${\rm div\,}b=0$) and, more generally, 
vector fields that have singular divergence ${\rm div\,}b \in L^1_{\loc}$. Parabolic equations with such vector fields, which are in the focus of this paper, arise in hydrodynamics and other physical applications.
We study one of the central questions in the regularity theory of \eqref{eq1}: what are the minimal assumptions on $b$ so that the heat kernel of \eqref{eq1} admits upper and/or lower Gaussian bounds? Such bounds, once established, open up the way to proving a number of other regularity results for \eqref{eq1}.
\medskip

The present paper continues \cite{KiS_MathAnn} where we dealt with time-homogeneous $a$, $b$ and investigated how the lower and the upper Gaussian heat kernel bounds depend on the divergence of $b$. This paper also continues  \cite{S} where two-sided Gaussian bounds for divergence-free time-inhomogeneous $b$ were established. 

\medskip

In \cite{KiS_MathAnn} and \cite{S}, the proof of two-sided Gaussian bounds uses a condition on $|b|:=(\sum_{k=1}^d b_k^2)^{1/2}$ (see Definition \ref{def2}). This condition, although quite general, does not allow to take into account possible oscillations in vector field $b$. In the present paper we handle oscillations in $b$ and also allow $b$ to have time-inhomogeneous singular divergence. This required a comprehensive modification of the argument from the two cited papers.
Possible oscillations of vector fields are taken into account using the class $L^\infty{\rm BMO}^{-1}$ of divergence-free vector fields whose components are 
distributional derivatives of functions of bounded mean oscillation 
(we recall the definition in the next section). The  results on the regularity theory of equation \eqref{eq1} with $b \in L^\infty{\rm BMO}^{-1}$ include, in particular, the parabolic Harnack inequality established, among other results, in Seregin-Silvestre-\v{S}verak-Zlatos \cite{SSSZ}, and global in time two-sided Gaussian bounds proved by Qian-Xi in \cite{QX}. Their Gaussian bounds in the a priori form follow as a special case of our Theorems \ref{thm0_}, \ref{thm0__}. Earlier, Osada \cite{O} proved two-sided Gaussian  heat kernel bounds for $b$ in $L^\infty(L^{-1}_\infty) \subset L^\infty{\rm BMO}^{-1}$, where $L_\infty^{-1}$ is the class of divergence-free vector fields whose components are 
distributional derivatives of functions in $L^\infty$.

\medskip

Let us now state our hypothesis on $b$ and describe our results. For notations, see Section \ref{notations_sect}.

\begin{definition}
\label{def1}
We say that a distributional vector field $b:\mathbb R_+ \rightarrow [\mathcal S']^d$ belongs to class $\mathbf{M}_\delta$ if for a.e.\,$t \in \mathbb R_+$
\begin{equation}
\label{b_cond}
|\langle b(t)\psi,\psi\rangle |  \leq \delta \|\nabla \psi\|_2\|\psi\|_2 + g(t)\|\psi\|_2^2
\end{equation}
for all $\psi$ in the L.\,Schwartz class $\mathcal S$,
for a non-negative function  $g$ that satisfies, on every finite interval $[0,T]$,
\begin{equation}
\label{g_cond}
\int_s^t g(\tau) d\tau \leq c\sqrt{t-s}, \quad s,t \in [0,T]
\end{equation}
for a constant $c=c_{\delta,T}$ (written as $b \in \mathbf{M}_\delta$, ``multiplicative class'').

\end{definition}

The multiplicative class $\mathbf{M}_\delta$ will also appear in another, more restrictive form: 

\begin{definition}
\label{def2}
A vector field $b \in [L^1_{\loc}(\mathbb R_+ \times \mathbb R^d)]^d$ belongs to class $\mathbf{MF}_\delta$ if for a.e.\,$t \in \mathbb R_+$
\begin{equation}
\label{hat_b}
\langle |b(t)|\psi,\psi\rangle   \leq \delta \|\nabla \psi\|_2\|\psi\|_2 + g(t)\|\psi\|_2^2, \quad \psi\in W^{1,2}
\end{equation}
under the same assumption on $g$ as in the previous definition (written as $b \in \mathbf{MF}_\delta$, multiplicatively form-bounded vector fields). 
\end{definition}

Note that both classes $\mathbf{M}_\delta$ and $\mathbf{MF}_\delta$ are closed with respect to addition (up to change of $\delta$).

\begin{examples}
The following vector fields $b$ are in $\mathbf{MF}_\delta$ for appropriate $\delta$.

1.~Vector fields in the critical Ladyzhenskaya-Prodi-Serrin class:
\begin{equation}
\label{LPS_c}
\tag{${\rm LPS}_c$}
|b| \in L^s(\mathbb R_+,L^r), \quad \frac{d}{r}+\frac{2}{s} \leq 1, \quad 2 \leq s \leq \infty, \;\;d \leq r \leq \infty.
\end{equation}

2. Vector fields $b$ that belong, uniformly in $t \in \mathbb R$, to the scaling-invariant Morrey class $M_{1+\varepsilon}$ with $\varepsilon>0$ fixed arbitrarily small:
\begin{equation}
\label{morrey}
\sup_{t \in \mathbb R}\|b(t)\|_{M_{1+\varepsilon}}=\sup_{t \in \mathbb R}\sup_{r>0, x \in \mathbb R^d} r\biggl(\frac{1}{|B_r(x)|}\int_{B_r(x)}|b(t,y)|^{1+\varepsilon}dy \biggr)^{\frac{1}{1+\varepsilon}}<\infty
\end{equation}
where $B_r(x)$ is the ball of radius $r$ centered at $x$, and $\varepsilon$ is fixed arbitrarily small.

See proofs of inclusions in Appendix \ref{examples_sect}.

3.~If $g \in L^2_{\loc}(\mathbb R_+)$, then condition \eqref{g_cond} obviously holds. 
Another example is $$g(t)=|t|^{-\frac{1}{2}},$$ which still satisfies \eqref{g_cond} 
but is not in $L_\loc^2(\mathbb R_+)$. 

\medskip

The function $g$ describes singularities of the vector field $b$ in the time variable.
In particular, if $b$ is bounded in the spatial variables, then condition \eqref{g_cond} says that $|b|$ can have at most $t^{-\frac{1}{2}}$ singularities in time,
so the
speed at which the streaming induced by $b$ displaces the solution is comparable to the speed at which the viscosity diffuses it. In this regard, see Carlen-Loss \cite{CL} who showed that the velocity fields in the Burgers and 2D Navier-Stokes vorticity equations have at most $t^{-1/2}$ time singularities. 
Let us add that if one is dealing with conditions on $b$ that are invariant with respect to the parabolic scaling, such as $b \in \mathbf{M}_\delta$ (so that there is a hope of arriving at the regularity results for parabolic equation \eqref{eq1} comparable to those for the heat equation, including two-sided Gaussian bounds), then one cannot extend $\|b(t)\|_\infty \leq c t^{-1/2}$ to $\|b(t)\|_\infty \leq c t^{-1/2-\varepsilon}$ regardless of how small one fixes $\varepsilon>0$ without destroying the parabolic scaling. See also \cite{QX2} in this regard.
\end{examples}

In Theorem \ref{thm0_}, we prove a priori upper Gaussian bound on the heat kernel of \eqref{eq1} assuming 
that $b \in \mathbf{M}_\delta$ for some finite $\delta$ and that the positive part $({\rm div\,}b)_+$ of divergence ${\rm div\,}b$ has sufficiently small Kato norm (cf.\,\eqref{kato}). Here ``a priori'' refers to the fact that Theorem \ref{thm0_} is proved for smooth $a$, $b$, but the constants in the upper Gaussian bound do not depend on the smoothness of $a$ and $b$.

\medskip

In Theorem \ref{thm0__}, we prove lower Gaussian bound on a heat kernel of \eqref{eq1} under more restrictive assumption on the vector field: 
$$
b=\hat{b}+\tilde{b}
$$
where  $\hat{b} \in \mathbf{MF}_\delta$ for some $\delta<\infty$ and $g \in L^2(\mathbb R_+)$, with ${\rm div\,}\hat{b}$ having sufficiently small Kato norm, and divergence-free $\tilde{b}$ is in the class $L^\infty {\rm BMO}^{-1}$.

\medskip

The fact that the assumptions of Theorem \ref{thm0_} are broader than the assumptions of Theorem \ref{thm0__} is seen from the straightforward inclusion $\mathbf{MF}_\delta \subset \mathbf{M}_\delta$ and the following proposition.

\begin{proposition}
\label{emb_prop}
If $\tilde{b} \in 
L^\infty{\rm BMO}^{-1},$ then $\tilde{b} \in \mathbf{M}_\delta$. 
\end{proposition}

The class $\mathbf{M}_\delta$ is the largest class of vector fields considered in this paper.

\medskip

The question of uniqueness of the heat kernel of \eqref{eq1} under the assumption $b \in \mathbf{M}_\delta$ is non-trivial. It is known that the heat kernel is unique if vector field $b$ belongs to $L^\infty{\rm BMO}^{-1}$,  a particular sub-class of $\mathbf{M}_\delta$. This is a result of Qian-Xi \cite{QX}, which they proved using the Lions' approach in the standard for equation $\eqref{eq1}$ triple of Hilbert spaces $W^{1,2} \subset L^2 \subset W^{-1,2}$. See detailed statement below. In this paper we single out another important sub-class of $\mathbf{M}_\delta$ for which one can prove uniqueness of the heat kernel, although only in the case $a=I$.

\begin{definition}
A vector field $b \in [L^1_{\loc}(\mathbb R_+ \times \mathbb R^d)]^d$ is said to be weakly form-bounded  (written as $b \in L^\infty\,\mathbf{F}_\delta^{\scriptscriptstyle 1/2}$) if, for a.e.\,$t \in \mathbb R_+$, 
\begin{equation}
\label{wfbd}
\||b(t)|^\frac{1}{2}\psi\|_2 \leq \sqrt{\delta} \|(\lambda-\Delta)^{\frac{1}{4}}\psi\|_2
\end{equation}
for all $ \psi \in \mathcal W^{\frac{1}{2},2}:=(\lambda-\Delta)^{-\frac{1}{4}}L^2$ (Bessel potential space)
for some constant $\lambda=\lambda_{\delta}\geq 0$. 
\end{definition}

We have 
\begin{equation}
\label{incl}
L^\infty\,\mathbf{F}_\delta^{\scriptscriptstyle 1/2} \subset \mathbf{MF}_\delta.
\end{equation}
Indeed, if $b\in L^\infty\,\mathbf{F}_\delta^{\scriptscriptstyle 1/2}$ and $\psi\in W^{1,2}$ then
\begin{align*}
\langle|b(t)|\psi,\psi\rangle &\leq \delta \langle (\lambda-\Delta)^{\frac{1}{4}}\psi,(\lambda-\Delta)^{\frac{1}{4}}\psi \rangle =\delta\langle(\lambda-\Delta)^\frac{1}{2}\psi,\psi\rangle\\
& \leq \delta \|(\lambda-\Delta)^{\frac{1}{2}}\psi\|_2\|\psi\|_2 = \delta\sqrt{\|\nabla \psi\|_2^2 + \lambda\|\psi\|_2^2}\|\psi\|_2 \leq \delta \|\nabla \psi\|_2\|\psi\|_2+\delta\sqrt{\lambda}\|\psi\|_2, 
\end{align*}
so $b \in \mathbf{MF}_\delta$ with $g(t)=\delta\sqrt{\lambda}$.

\medskip

The classes of vector fields mentioned in Examples 1 and 2 above are all contained in $L^\infty\,\mathbf{F}_\delta^{\scriptscriptstyle 1/2}$.

\medskip

The class $L^\infty\,\mathbf{F}_\delta^{\scriptscriptstyle 1/2}$ is quite different from  $L^\infty{\rm BMO}^{-1}$. Indeed, the proof of uniqueness of the heat kernel of the parabolic equation
\begin{equation}
\label{eq0}
(\partial_t - \Delta + b \cdot \nabla)u=0
\end{equation}
with $b \in L^\infty\,\mathbf{F}_\delta^{\scriptscriptstyle 1/2}$ requires one to run Lions' approach in a non-standard triple of Hilbert spaces $$\mathcal W^{\frac{3}{2},2} \subset \mathcal W^{\frac{1}{2},2} \subset \mathcal W^{-\frac{1}{2},2}.$$
In particular, the solution of a Cauchy problem for \eqref{eq0} satisfies different energy inequalities depending on whether $b$ is in  $L^\infty\,\mathbf{F}_\delta^{\scriptscriptstyle 1/2}$ or in $L^\infty{\rm BMO}^{-1}$, cf.\,\eqref{ei2} and \eqref{ei0_bmo}. 
Put another way, the Sobolev embedding properties of $-\Delta+b \cdot \nabla$ change
 drastically as one transitions from  $L^\infty {\rm BMO}^{-1}$ to $L^\infty \mathbf{F}_\delta^{\scriptscriptstyle 1/2}$.
This allows us to conclude that these two subclasses of $\mathbf{M}_\delta$ are quite far apart.

\medskip

Let us emphasize that the lower bound in 
Theorem \ref{thm0__} is new even in the case $a=I$, $b=b(x)$.

\medskip

Concerning the lower and upper Gaussian bounds for time-homogeneous singular drifts having singular divergence, we refer, in addition to \cite{KiS_MathAnn}, to earlier results by Liskevich-Zhang \cite{LZ} who also considered form-boundedness and Kato class conditions on $|b|$ and ${\rm div\,}b$. See detailed comparison of the results in \cite[Sect.\,1]{KiS_MathAnn}.

\medskip

The proof of the upper bound in Theorem \ref{thm0_} uses the Moser iteration method. The proof of the lower bound in Theorem \ref{thm0__} uses a rather non-standard version of Nash's method \cite{N}. 
More precisely, the assumption $b \in L^\infty\mathbf{F}_\delta^{\scriptscriptstyle 1/2}$ prohibits, even if ${\rm div\,}b=0$,  the use of quadratic inequalities in the analysis of the Nash $G$-functions in the proof of the a priori lower Gaussian  bound (Theorem \ref{thm2}). As a consequence, one needs to use a relatively sophisticated regularization of the Nash $G$-functions. 
Arguably, $L^\infty\mathbf{F}_\delta^{\scriptscriptstyle 1/2}$ is the ``hard part'' of $\mathbf{M}_\delta$.

\medskip

In this paper, the proof of the lower bound uses the upper bound. We mention, however, \cite[Theorem 1]{KiS_MathAnn}, where it was demonstrated that these bounds are, in principle, independent. Namely, a lower Gaussian bound on the heat kernel of \eqref{eq1} (with time-homogeneous coefficients) holds provided that $b$ is form-bounded with $\delta<4\sigma^2$ and ${\rm div\,}b \geq 0$ (in the sense of distributions), in which case an upper Gaussian bound is in general invalid. Conversely, there are situations where an upper Gaussian bound holds but there are no lower Gaussian bounds, see \cite{KiS_MathAnn} for details.

\medskip

It should be added that we prove global in time heat kernel bounds, so our conditions on $g$ are global, but it is straightforward to make them local in time with local conditions on $g$ as above.

\medskip

Let us note that the classes of singular vector fields discussed in this paper are scaling-invariant. It is known that one can venture beyond the scaling-invariance, considering \eqref{eq0} with $b \in L^2_{\loc}(\mathbb R_+ \times \mathbb R^d)$, ${\rm div\,}b=0$ such that,  for some
$0<a \leq 1$ and $\delta<\infty$, 
\begin{align*}
 \int_0^\infty \langle |b(t,\cdot)|^{1+a}\varphi^2(t,\cdot)\rangle dt   \leq \delta \int_0^\infty\|\nabla \varphi(t,\cdot)\|_2^2 dt, \quad \forall \varphi \in C_c^\infty(\mathbb R_+ \times \mathbb R^d),
\end{align*}
containing e.g.\,zero-divergence $b=b(x)$ with $|b| \in L^{p}$, $p>\frac{d}{2}$ (essentially, twice more singular than \eqref{eq0}), see \cite{Z1}. See also \cite{QX2}. Although in this case one has to sacrifice much of the regularity theory of \eqref{eq0} and \eqref{eq1}, some results can be salvaged. This includes boundedness of weak solutions,  a non-Gaussian upper bound, see cited papers for details.

\subsection{Notations and auxiliary results}
\label{notations_sect}
Let $L^p=L^p(\mathbb R^d)$, $p\geq 1$ denote the standard Lebesgue space with norm $\|\cdot\|_p$, $W^{1,p}=W^{1,p}(\mathbb R^d)$ the Sobolev spaces, $\mathcal S'=\mathcal S'(\mathbb R^d)$ the space of Schwartz distributions.

Let $\mathcal B(X,Y)$ be the space of bounded linear operators between Banach spaces $X \rightarrow Y$  with operator norm $\|\cdot\|_{X \rightarrow Y}$. Let $\mathcal B(X):=\mathcal B(X,X)$. Set $\|\cdot\|_{p \rightarrow q}:=\|\cdot\|_{L^p \rightarrow L^q}$. 

Given a $d \times d$ matrix $P=(P_{ij})_{i,j=1}^d$ with entries in $X$, we set $\|P\|_X:=\bigl(\sum_{i,j=1}^d \|P_{ij}\|_X^2\bigr)^{\frac{1}{2}}$.

Put$$
\langle f\rangle:=\int_{\mathbb R^d}fdx, \quad \langle f,g\rangle:=\langle f\bar{g}\rangle.
$$

The following class is well known:

\begin{definition*}
A vector field $b:\mathbb R_+ \rightarrow [\mathcal S']^d$ with ${\rm div\,}b=0$ is said to be in the class $L^\infty\,{\rm BMO}^{-1}$ if
\begin{equation}
\label{b_repr}
b_k(t)=\sum_{i=1}^d \nabla_i B_{ik}(t), \quad t \in \mathbb R_+,
\end{equation}
for some skew-symmetric matrix $B$ with entries $B_{ik} \in L^\infty {\rm BMO} \equiv L^\infty(\mathbb R_+,{\rm BMO})$.
\end{definition*}

The class $L^\infty\, {\rm BMO}^{-1}$ is endowed with semi-norm
$$
\|b\|_{L^\infty\, {\rm BMO}^{-1}} = \|B\|_{L^\infty\, {\rm BMO}} :=\sup_{t \in \mathbb R_+}\|B(t)\|_{\rm BMO}.
$$
Here ${\rm BMO}$ is the space of functions of bounded mean oscillation on $\mathbb R^d$; recall that a function $F \in L^1_{\loc}(\mathbb R^d)$ is of bounded mean oscillation if 
$$
\|F\|_{\rm BMO}:= \sup_{Q}\frac{1}{|Q|}\int_{Q}|F-\bar{F}|dx<\infty, \quad \text{where }\bar{F}:=\frac{1}{|Q|}\int_Q F dx.
$$
with the supremum taken over all cubes $Q \subset \mathbb R^d$ with sides parallel to the axes, $|Q|$ is the volume of $Q$.

If $b$ is time independent, then we write simply $b \in {\rm BMO}^{-1}$. (Occasionally, we will be adding ``${\rm div\,}b=0$'' to make the paper easier to follow although this is redundant.)

As it was
 demonstrated in \cite{QX}, the functions $B_{ik}$ can always be modified to be in $L^q_{\loc}(\mathbb R_+ \times \mathbb R^d)$ for all $1 \leq q<\infty$ (by adding functions that only depend on $t$). We assume in what follows that this modification has been made.

\begin{proposition}[{\cite[Theorem 4]{CLMS}}]
\label{lem_bmo_est}
Let $b \in {\rm BMO}^{-1}$, ${\rm div\,}b=0$.
Then, for all $u,v \in W^{1,2}$, $$|\langle b \cdot\nabla u,v\rangle|\leq \|B\|_{\rm BMO}\|\nabla u\|_2\|\nabla v\|_2.$$
\end{proposition}

\begin{proposition} [{\cite[Prop.\,3.2]{QX}}]
\label{qian_xi_prop}
Given a $f\in W^{1,2}$, one has $|f|\nabla_i|f|\in \mathcal H_1$, where $\mathcal H_1$ is the real Hardy space, and
$$
\||f|\nabla_i |f|\|_{\mathcal H_1} \leq C\|\nabla f\|_2\|f\|_2, \quad C=C(d).
$$
\end{proposition}

 \begin{definition*}
${\rm div\,}b \in L^1_{\loc}(\mathbb R_+ \times \mathbb R^d)$ is said to be form-bounded (with form-bound $\nu<\infty$) if
\begin{equation}
\label{fbd_potential}
\big\||{\rm div\,}b(t)|^{\frac{1}{2}} \psi\big\|_2^2   \leq \nu \|\nabla \psi\|^2_2  + h(t)\|\psi\|_2^2 \quad \forall \psi \in W^{1,2}
\end{equation}
for a.e.\,$t \in \mathbb R_+$,
for some function $0 \leq h  \in L^1_{\loc}(\mathbb R_+)$.
\end{definition*}

\bigskip

\section{Main results}

\medskip

We start with the basic results of the well-posedness of Cauchy problem for equations \eqref{eq0} and \eqref{eq1} without any assumptions on ${\rm div\,}b$.

Instead of \eqref{eq0}, it will be convenient to work with equation
\begin{equation}
\label{eq0_}
(\partial_t + \lambda- \Delta + b \cdot \nabla)u=0,
\end{equation}
where $\lambda$ is from the condition $b \in L^\infty\,\mathbf{F}_\delta^{\scriptscriptstyle 1/2}.$
In this regard, we introduce the 
scale of Bessel potential spaces $\mathcal W^{\alpha,2}$ endowed with the norm $$\|v\|_{\mathcal W^{\alpha,2}}:=\|(\lambda -\Delta)^\frac{\alpha}{2}v\|_2.$$ 

\medskip

Assertion (\textit{ii}) in theorem below in the case case $b \in L^\infty\,{\rm BMO}^{-1}$ is due to \cite[Theorem 5.2]{QX}.  We included it for the sake of completeness. Regarding the elliptic setting, see \cite[Theorem 3.1]{Zh}.

\begin{theorem}
\label{thm0}
Let $d \geq 3$, $T>0$. The following is true:

{\rm(\textit{i})} Let $b$ be weakly form-bounded: $$b \in L^\infty\,\mathbf{F}_\delta^{\scriptscriptstyle 1/2} \quad \text{ with }\delta<1 \quad \text{\rm (see \eqref{wfbd})}.$$
Then for every $f \in \mathcal W^{\frac{1}{2},2}$ there exists a unique weak solution to Cauchy problem for \eqref{eq0_} with initial condition $u(s+)=f$, i.e.\,a unique
 in $L^\infty_{\loc}(]s,T[,\mathcal W^{\frac{1}{2},2}) \cap L^2_{\loc}(]s,T[,\mathcal W^{\frac{3}{2},2})$ function $u$  satisfying 
\begin{align}
 \int_{s}^{T} \langle (\lambda-\Delta)^{\frac{1}{4}}u,\partial_t (\lambda-\Delta)^{\frac{1}{4}}\varphi \rangle dt & = \int_{s}^{T} \langle (\lambda-\Delta)^{\frac{3}{4}}u,(\lambda-\Delta)^{\frac{3}{4}}\varphi \rangle dt  \notag
\\
& + \int_{s}^{T} \langle b(t) \cdot \nabla u,(\lambda-\Delta)^{\frac{1}{2}}\varphi\rangle \label{weak_sol}
\end{align}
for all $\varphi \in C_c^\infty(]s,T[,\mathcal S)$ and
\begin{align}
\label{cp}
w\mbox{-}{\mathcal W^{\frac{1}{2},2}}\mbox{-}\lim_{t \downarrow s}u(t)=f.
\end{align}
 Furthermore, $u \in C([s,T],\mathcal W^{\frac{1}{2},2})$, and the following energy inequality holds:
\begin{equation}
\label{ei2}
\|u(t)\|_{\mathcal W^{\frac{1}{2},2}}^2 + 2(1-\delta)\int_{s}^{t} \|u(\tau)\|^2_{\mathcal W^{\frac{3}{2},2}}d\tau \leq  \|f\|_{\mathcal W^{\frac{1}{2},2}}^2, \quad 0 \leq s < t \leq T.
\end{equation}
The operators $T^{t,s}f(x):=u(t,s,x)$ constitute a contraction strongly continuous Markov evolution family in $\mathcal W^{\frac{1}{2},2}$.
If $\{b_\varepsilon\}_{\varepsilon>0}$ is a family of bounded smooth vector fields such that $b_\varepsilon \in L^\infty\,\mathbf{F}_\delta^{\scriptscriptstyle 1/2}$ with the same $\lambda$ as $b$, 
$
b_\varepsilon \rightarrow b$ in $[L^1_{\loc}(\mathbb R_+ \times \mathbb R^d)]^d$ as $\varepsilon \to 0
$
(see example of such vector fields in Proposition \ref{fbd_approx_prop}), and
if $u_\varepsilon$ denotes the  solution to Cauchy problem \eqref{eq0_}, \eqref{cp} with the vector field $b_\varepsilon$, then
$$
u_\varepsilon \rightarrow u \quad \text{weakly in } L^2([s,T],\mathcal W^{\frac{3}{2},2}) \text{ as } \varepsilon \to 0.
$$

\medskip

{\rm(\textit{ii})} If either ($|b| \in L^2_{\loc}(\mathbb R_+ \times \mathbb R^d)$, ${\rm div\,}b$ is form-bounded with form-bound $\nu<1$) or   ($b \in L^\infty\,{\rm BMO}^{-1}$, ${\rm div\,}b=0$), then for every $f \in L^2$ there exists a unique weak solution $u$ to the corresponding Cauchy problem for \eqref{eq1} (in the standard triple $W^{1,2} \subset L^2 \subset W^{-1,2}$), which satisfies the classical energy inequality
\begin{equation}
\label{ei0_bmo}
\|u(t)\|_{2}^2 + c\int_{s}^{t} \|\nabla u(\tau)\|^2_{2}d\tau \leq  \|f\|_{2}^2, \quad c>0, \quad 0 \leq s<t.
\end{equation}

\end{theorem}

The notion of weak solution to equation \eqref{eq0_} in assertion (\textit{i})
is obtained by formally multiplying \eqref{eq0_} by test function $(\lambda-\Delta)^{\frac{1}{2}}\varphi$ and integrating.
The last term in \eqref{weak_sol} is well defined by Proposition \ref{lem_fbd_est}.

 The proof of Theorem \ref{thm0}(\textit{i}) uses the Lions variational approach in the triple of Bessel potential spaces $\mathcal W^{\frac{3}{2},2} \subset \mathcal W^{\frac{1}{2},2} \subset \mathcal W^{-\frac{1}{2},2}.$
 Theorem \ref{thm0}(\textit{i}) can be viewed as the first step towards a regularity theory of \eqref{eq0} with weakly form-bounded $b$. In the time-homogeneous case $b=b(x)$, the class \eqref{wfbd} provides sharp $L^p \rightarrow L^q$ bounds on the corresponding to \eqref{eq0} semigroup \cite{S}, a detailed Sobolev regularity theory of elliptic operator $-\Delta + b \cdot \nabla$ in $L^p$ for $p$ large and the corresponding Feller semigroup \cite{Ki}, see discussion in Section \ref{comm_sect}. The latter determines, for every initial point $x \in \mathbb R^d$, a ``sequentially unique'' weak solution to the SDE 
\begin{equation}
\label{sde}
X_t=x - \int_0^t b(X_s)ds + \sqrt{2}B_t, \quad t \geq 0,
\end{equation}
where $B_t$ is the standard $d$-dimensional Brownian motion, $x \in \mathbb R^d$ is the initial point, see \cite{KiS_BM}. (See, however, recent developments in \cite{Ki2} regarding time-inhomogeneous $b$.)

\begin{remark}
\label{rem_add_g}
One can extend the definition of the class $L^\infty\,\mathbf{F}_\delta^{\scriptscriptstyle 1/2}$ by considering $b \in [L^1_{\loc}(\mathbb R_+ \times \mathbb R^d)]^d$ such that
\begin{equation}
\label{weak_fbd_g}
\||b(t)|^{\frac{1}{2}}\psi\|_2^2 \leq  \delta \|(\lambda-\Delta)^{\frac{1}{4}}\psi\|_2^2 + g(t)\|\psi\|_2^2, \quad \psi \in (\lambda-\Delta)^{-\frac{1}{4}}L^2,
\end{equation}
where  $0 \leq g \in L^2_{\loc}(\mathbb R_+)$, see Appendix \ref{appg} for details. 
\end{remark}

Next, we turn to the question of what assumptions on locally unbounded $b$ provide upper and lower Gaussian bounds on the heat kernel of  equation \eqref{eq1}. Compared to the previous theorem, we will weaken the assumption on the vector field $b$ even further to $b \in \mathbf{M}_\delta$, in particular taking into account  possible cancellations, but requiring the existence of ${\rm div\,}b \in L^1_{\loc}(\mathbb R_+ \times \mathbb R^d)$ and, moreover, the smallness of the Kato norms of its positive and/or negative parts $({\rm div\,}b)_+$,  $({\rm div\,}b)_-$. These assumptions allow $b$ and ${\rm div\,}b$ to be quite singular.

Put
\begin{align*}
k_c(t,x,y) & \equiv k(c t,x,y) :=(4\pi c t)^{-\frac{d}{2}}e^{ - \frac{|x-y|^2}{4c t}}, \quad c>0.
\end{align*}

\begin{theorem}
\label{thm0_}

Let $d \geq 3$. Let $a \in H_{\sigma,\xi}$, $b \in \mathbf{M}_\delta$ with multiplicative bound $\delta<\infty$ and function $g$ satisfying
\begin{equation}
\label{g_cond2}
\int_s^t g(\tau) d\tau \leq c_\delta\sqrt{t-s}, \quad 0 \leq s<t<\infty
\end{equation}
for some constant $c_\delta \geq 0$. Let $\mu_+$ denote the global in time Kato norm of ``potential'' $({\rm div\,}b)_+$, i.e.\,
the maximum between
\begin{equation}
\label{kato}
\sup_{t \geq 0, x \in \mathbb R^d}\int_{0}^{t} \langle k(t-\tau,x,\cdot)({\rm div\,}b)_+(\tau,\cdot)\rangle d\tau, \quad
\sup_{s \geq 0, x \in \mathbb R^d}\int_{s}^{\infty} \langle k(\tau-s,x,\cdot)({\rm div\,}b)_+(\tau,\cdot)\rangle d\tau.
\end{equation}
Also, assume that $a$, $b$, ${\rm div\,}b$ are bounded smooth. 

The following is true.
If $\mu_+$ is smaller than a certain generic constant (i.e.\,a constant that depends only on $d$, $\sigma$, $\xi$, $\delta$, $c_\delta$, but not on the smoothness of $a$ or the boundedness and smoothness of $b$, ${\rm div\,}b$), then the heat kernel $u(t,x;s,y)$ of equation \eqref{eq1} satisfies
a global in time upper Gaussian bound
\[
u(t,x;s,y)\leq c_3 k_{c_4}(t-s;x-y) \quad \text{for all $0 \leq s<t<\infty$, $x,y \in \mathbb R^d$}
\]
with generic constants  $c_i$, $i=3,4$ that can also depend on $\mu_+$.
\end{theorem}

If the Kato norm $\mu_+$ of $ ({\rm div\,}\hat{b})_{+}$ is finite, as in the theorem above, then we say that $ ({\rm div\,}\hat{b})_{+}$ belongs to the Kato class.

\begin{remark}
Regarding the above condition on the smallness of the Kato norm $\mu_+$ of $({\rm div\,}b)_+$, recall that  if $({\rm div\,}b)_+ \in \mathbf{1}_{\{0 \leq t \leq 1\}}L^\infty(\mathbb R_+,L^p)$, $p>\frac{d}{2}$, then $\mu_+$ can be chosen arbitrarily small. There also exist $({\rm div\,} b)_+(x)$ that are not even in $L^p_{\loc}$ for any $p>1$ yet have $\mu_+$ finite (or sufficiently small upon multiplying $b$ by a small constant).
\end{remark}

In the next theorem we assume that $a \in H_{\sigma,\xi}$ and $b$ are only measurable, but $b$ satisfies a more restrictive condition:
\begin{equation}
\label{c1}
b=\tilde{b} + \hat{b},
\end{equation}
 where
\begin{equation}
\label{c2}
\tilde{b} \in L^\infty\,{\rm BMO}^{-1}, \;\;{\rm div\,}\tilde{b}=0,
\end{equation}
\begin{equation}
\label{c3}
\hat{b} \in \mathbf{MF}_{\delta} \text{ with } g \in L^2(\mathbb R_+), 
\end{equation}
\begin{equation}
\label{c4}
({\rm div\,}\hat{b})_{\pm} \in L^1_{\loc}(\mathbb R_+ \times \mathbb R^d) \text{ belong to the Kato class with Kato norms $\mu_\pm$, respectively, and }
\end{equation}
\begin{equation}
\label{c5}
\begin{array}{ll}
|({\rm div\,}\hat{b})_{\pm}| \text{ are form-bounded: }& \big\langle \big|({\rm div\,}\hat{b})_{\pm}(t)\big|\psi,\psi\big\rangle  \leq \nu_\pm \|\nabla \psi\|_2^2 + h_\pm(t)\|\psi\|_2^2, \quad \forall\psi \in W^{1,2}\\
& \text{for a.e.} t \in \mathbb R_+ \text{ for some } 0 \leq h_\pm \in L^1(\mathbb R_+).
\end{array}
\end{equation}

\begin{definition}
Assume that \eqref{c1}-\eqref{c5} hold.
We say that $v \in L^2([s,T],W^{1,2})$ is an approximation solution to Cauchy problem for equation \eqref{eq1} with initial condition 
\begin{equation}
\label{approx_cond}
s\mbox{-}L^2\mbox{-}\lim_{t\downarrow s}v(s)=f \in L^2
\end{equation}
 if
$$
v=w\mbox{-}L^2_{\loc}([s,\infty[,W^{1,2})\mbox{-}\lim_{\varepsilon \downarrow 0}\lim_{\varepsilon_1 \downarrow 0}v_{\varepsilon_1,\varepsilon}, \quad 
$$
where $v_{\varepsilon_1,\varepsilon}$ solve
$$
(\partial_t - \nabla \cdot a_{\varepsilon_1} \cdot \nabla  + b_\varepsilon \cdot \nabla )v_{\varepsilon_1,\varepsilon}=0, \quad v_{\varepsilon_1,\varepsilon}(s)=f
$$
for some bounded smooth
$a_{\varepsilon_1} \in H_{\sigma,\xi}$, $\hat{b}_\varepsilon$ and ${(\rm div\,}\hat{b})_{\pm,\varepsilon} \geq 0$ that have  the same multiplicative bound $\delta$, $g$ and the form-bounds $\nu_\pm$, the Kato norms $\mu_\pm$, respectively, and bounded smooth 
$\tilde{b}_\varepsilon=\nabla B_\varepsilon$ with the property $$\|\tilde{b}_\varepsilon \|_{L^\infty\,{\rm BMO}^{-1}} \leq C\|\tilde{b} \|_{L^\infty\,{\rm BMO}^{-1}}$$ for a constant $C$ independent of $\varepsilon$ and $B_\varepsilon$ bounded smooth skew-symmetric matrices with entries in $L^\infty\,{\rm BMO}$,
such that
\begin{equation}
\label{all_conv}
\begin{array}{rl}
a_{\varepsilon_1} \rightarrow a \quad & \text{  in } \quad [L^2_{\loc}(\mathbb R_+ \times \mathbb R^d)]^{d \times d}, \\
\hat{b}_\varepsilon \rightarrow \hat{b}  \quad & \text{  in } \quad [L^1_{\loc}(\mathbb R_+ \times \mathbb R^d)]^{d}, \\
B_{\varepsilon} \rightarrow B \quad & \text{ in }\quad  [L^2_{\loc}(\mathbb R_+ \times \mathbb R^d)]^{d \times d}, \\
{(\rm div\,}\hat{b})_{\pm,\varepsilon} \rightarrow {(\rm div\,}\hat{b})_\pm \quad & \text{  in } \quad L^1_{\loc}(\mathbb R_+ \times \mathbb R^d)
\end{array}
\end{equation}
as $\varepsilon_1$, $\varepsilon \downarrow 0$.
\end{definition}

It is easy to see that if $v$ is an approximation solution to \eqref{eq1}, then it is also a weak solution to \eqref{eq1}:
\begin{align}
\label{weak_sol_id}
-\int_s^t\langle v,\partial_r \varphi\rangle dr + \int_s^t \big\langle (a + B) \cdot \nabla v,\nabla \varphi \big\rangle dr & - \int_s^t \big\langle \hat{b}v,\nabla \varphi\big\rangle dr  - \int_s^t \big\langle ({\rm div\,}\hat{b})v,\varphi\big\rangle dr=0
\end{align}
for all $\varphi \in C_c^\infty(]s,t[,\mathcal S)$.
Here $\tilde{b}=\nabla B$ where $B$ is skew-symmetric in $L^\infty\,{\rm BMO} \cap L^2_{\loc}(\mathbb R_+ \times \mathbb R^d)$, see Section \ref{notations_sect}.

\begin{definition}
We call a constant generic$\ast$ if it depends only on $$d, \sigma, \xi, \delta, \|g\|_{L^2(\mathbb R_+)}, \nu_\pm, \|h_\pm\|_{L^1(\mathbb R_+)} \text{ and } \|\tilde{b}\|_{L^\infty\,{\rm BMO}^{-1}}.$$
\end{definition}

\begin{theorem}
\label{thm0__}

Let $d \geq 3$, $a \in H_{\sigma,\xi}$. There exist generic$\ast$ constants 
$\mu_\pm^\ast$ such that if \eqref{c1}-\eqref{c5} hold with $$\delta<\infty,  \quad \nu_\pm<\sigma \quad 
\text {and with Kato norms }\mu_\pm < \mu_\pm^\ast,$$ then there exists a H\"{o}lder continuous heat kernel $u(t,x;s,y)$ of equation \eqref{eq1} satisfying:

\smallskip

{\rm(a)} A global in time lower Gaussian bound 
$$
c_1 k_{c_2}(t-s;x-y) \leq u(t,x;s,y)
$$
holds, in addition to the upper bound in Theorem \ref{thm0_}, for all $0 \leq s<t<\infty$, $x$, $y \in \mathbb R^d$,
with generic$\ast$ constants $c_1$-$c_4$ that can also depend on $\mu_\pm$.

\medskip

{\rm (b)} The function $$v(t,x):=\big\langle u(t,x;s,\cdot)f(\cdot)\big\rangle, \quad f \in L^2,$$ is an approximation solution to Cauchy problem for \eqref{eq1}.

\medskip

{\rm (c)} The operators $T^{t,s}f:=\big\langle u(t,x;s,\cdot)f(\cdot)\big\rangle$ determine a quasi bounded strongly continuous Feller evolution family of integral operators in $\mathcal B(X)$, $X=L^p$, $1 \leq p <\infty$ or $X=C_u$. The heat kernel  $u(t,x;s,y)$ is defined as the integral kernel of these operators, possibly after a modification on a measure zero set.

\smallskip

\smallskip

\smallskip

{\rm (d)} If either 
\begin{equation}
\tag{${\rm d}_1$}
f \in L^2, \quad a \in H_{\sigma,\xi}, \quad  b=\tilde{b}
\end{equation}
or 
\begin{equation}
\tag{${\rm d}_2$}
f \in \mathcal W^{\frac{1}{2},2}, \quad a=I, \quad b=\hat{b}=\hat{b}^{(1)}+\hat{b}^{(2)}
\end{equation}
such that, for a.e.\,$t \in \mathbb R_+$,
$$
\||\hat{b}^{(1)}(t)|(\lambda-\Delta)^{-\frac{1}{2}}\|_{2 \rightarrow 2} \leq \sqrt{\delta_1}, \quad \|(\lambda-\Delta)^{-\frac{1}{2}}|\hat{b}^{(2)}(t)|\|_{\infty} \leq \sqrt{\delta_2}
$$
with $\sqrt{\delta_1}+\sqrt{\delta_2}<1$, then an approximation solution to Cauchy problem \eqref{eq1}, \eqref{approx_cond} is unique.
\end{theorem}

The last assertion in the case (${\rm d}_1$) is due to \cite{QX} (in fact, valid for weak solutions in the standard triple, not just approximation solutions). The case (${\rm d}_2$) is a consequence of Theorem \ref{thm0}(\textit{ii}). 

\medskip

If $f \in L^2$ and $|b| \in L^2_{\loc}(\mathbb R_+ \times \mathbb R^d)$, ${\rm div\,}b$ is form-bounded with form-bound $\nu<1$, then the corresponding Cauchy problem for \eqref{eq1} has a unique weak solution, cf.\,Theorem \ref{thm0}(\textit{ii}). Theorem \ref{thm0__}(d) has the
advantage that it does not require $|b| \in L^2_{\loc}(\mathbb R_+ \times \mathbb R^d)$.

\medskip

The constants $\mu_\pm^\ast$ that bound the admissible values of the Kato norms in Theorem \ref{thm0__} are given by, in principle, explicit but rather complicated expressions (e.g.\,they will depend on the constants in the upper Gaussian bound of Theorem \ref{thm0_}), so we will not attempt writing them down here.

\begin{remark}
The upper bound in Theorem \ref{thm0_}  becomes local in time if
$g$ satisfies only \eqref{g_cond},
$\mu_+$ by the local Kato norm of $({\rm div\,}b)_+$, i.e.\,the maximum between
$$
\inf_{\vartheta>0}\sup_{t \geq \vartheta, x \in \mathbb R^d}\int_{t-\vartheta}^{t} \langle k(t-r,x,\cdot)({\rm div\,}b)_+(r,\cdot)\rangle dr
$$
and
$$
\inf_{\vartheta>0}\sup_{s \geq 0, x \in \mathbb R^d}\int_{s}^{s+\vartheta} \langle k(r-s,x,\cdot)({\rm div\,}b)_+(r,\cdot)\rangle dr.
$$
The two-sided bound in (a) becomes local in time if one requires $g \in L^2_{\loc}(\mathbb R_+)$, $h_\pm \in L^1_{\loc}(\mathbb R_+)$ and replaces the global in time Kato norms of $({\rm div\,}b)_\pm$ by their local counterparts.
\end{remark}

\begin{remark}
\label{rem_Nash_class}
Concerning the divergence form equation \eqref{eq1}, the authors  obtained in \cite{KiS_nash}  an $L^1$ strong solution theory of \eqref{eq1} with measurable uniformly elliptic $a=a(x)$ and $b=b(x)$ in the elliptic Nash class 
$$
|b| \in L^2_{\loc} \quad \text{ and } \quad \sup_{x \in \mathbb{R}^d} \int_0^h \big\langle k(t,x,\cdot)|b(\cdot)|^2\big\rangle^{\frac{1}{2}}  \; \frac{dt}{\sqrt{t}} \quad \text{ is sufficiently small, for some $h>0$},
$$
without any assumptions on ${\rm div\,}b$.
These assumptions, moreover, provide two-sided Gaussian bounds on the heat kernel of \eqref{eq1}.

The elliptic Nash class contains e.g.\,$b=b(x)$ with $|b| \in L^p+L^\infty$, $p>d$, but it also contains some $b$ with $|b| \not \in L^{2+\varepsilon}_{\loc}$, $\varepsilon>0$.
Despite the fact that equation \eqref{eq1} with $b$ in the elliptic Nash class admits  $L^1$ strong solution theory, it does not seem to admit even an $L^2$ weak solution theory, see discussion in \cite{KiS_nash}.

\end{remark}

\begin{remark}

Parabolic equation \eqref{eq1} admits weak solution theory in the standard tripe of Hilbert spaces $W^{1,2} \subset L^2 \subset W^{-1,2}$ if either  $b \in L^\infty{\rm BMO}^{-1}$ \cite{QX} or $|b|$ is form-bounded with form-bound $\delta<1$:

\begin{definition*}
A vector field $b \in [L^2_{\loc}(\mathbb R_+ \times \mathbb R^d)]^d$ is said to be form-bounded (with form-bound $\delta>0$) if for a.e.\,$t \in \mathbb R_+$
\begin{equation}
\label{fbd_class}
\|b(t,\cdot) \psi\|_2^2   \leq \delta \|\nabla \psi\|_2^2+ g(t)\|\psi\|_2^2, \quad \forall\psi \in W^{1,2},
\end{equation}
for a function $g  \in L^1_{\loc}(\mathbb R_+)$.
\end{definition*}

For time-homogeneous $a$ and $b$,
Mazya-Verbitsky \cite{MV} proved that the general second order elliptic operator $-\nabla \cdot a \cdot \nabla + b \cdot \nabla + V$ with time-homogeneous coefficients is $W^{1,2} \rightarrow W^{-1,2}$ bounded if and only if 
$b=\hat{b}+\tilde{b}$, where $\hat{b}$ is form-bounded with some $\delta$, and $\tilde{b} \in {\rm BMO}^{-1}$,
and distributional $V$ is form-bounded in the sense that
$$
|\langle V\psi,\psi\rangle| \leq \nu \|\nabla \psi\|_2^2 + c(\nu)\|\psi\|_2^2, \quad \psi \in C_c^\infty
$$
for suitable $\nu>0$, $c(\nu)<\infty$.

The results of the present paper show that if one deals with such aspects of the regularity theory of $-\Delta +b \cdot \nabla$, $-\nabla \cdot a \cdot \nabla + b \cdot \nabla$ as the existence and uniqueness of weak solution, upper and lower Gaussian bounds, one can  consider even less restrictive  
assumptions on the vector field $b$.
Let us also add that the operator $-\Delta + b \cdot \nabla + V$ also admits an $L^2$ theory (but not  $W^{1,2} \rightarrow W^{-1,2}$ boundedness) if $b$ is only weakly form-bounded, i.e.\,$b \in \mathbf{F}_{\delta_b}^{\scriptscriptstyle 1/2}$, but the potential $V$ satisfies a somewhat more restrictive condition than the form-boundedness: $$\||V|^{\frac{3}{4}}(\lambda-\Delta)^{-\frac{3}{4}}\|_{2 \rightarrow 2} \leq \delta_V^{\frac{3}{4}},$$
where $\delta_b+\delta_V<1$. See \cite[Sect.\,5.4]{KiS_theory}.
\end{remark}

See also further discussion and examples in Section \ref{comm_sect}.

\bigskip

\section{Proof of Proposition \ref{emb_prop}}

\label{emb_prop_proof}

Given a $b \in L^\infty\,{\rm BMO}^{-1}$, 
we write
$$
\langle b_k(t) \psi,\psi\rangle = -2\sum_{i=1}^d \langle B_{ik}(t)|\psi|\nabla_i |\psi|\rangle.
$$
By the $\mathcal H_1$-${\rm BMO}$ duality, 
$$
|\langle b_k(t)\psi,\psi\rangle| \leq 2\sum_{i=1}^d\|B_{ik}(t)\|_{\rm BMO}\||\psi|\nabla_i |\psi|\|_{\mathcal H_1}.
$$
Therefore, by Proposition \ref{qian_xi_prop},
$$
|\langle b(t) \psi,\psi\rangle|\equiv \bigl(\sum_{k=1}^d \langle b_k(t) \psi,\psi\rangle^2 \bigr)^{1/2} \leq 2C\|B(t)\|_{\rm BMO}\|\nabla \psi\|_2\|\psi\|_2,
$$
i.e.\,$b \in \mathbf{M}_\delta$ with $\delta=2C\|B\|_{L^\infty\,{\rm BMO}}$ and $g=0$.

\bigskip

\section{Proof of Theorem \ref{thm0}(\rm{\textit{i}})}

Put $$H_+:=\mathcal W^{\frac{3}{2},2}, \quad H:=\mathcal W^{\frac{1}{2},2}, \quad H_-:=\mathcal W^{-\frac{1}{2},2}.$$ Clearly, $H_-$ is the dual of $H_+$ with respect to the inner product in $H$. 

\begin{proposition}
\label{lem_fbd_est}
Let $u \in L^2_{\loc}(]s,T[,H_+)$. Then, for every 
$\varphi \in L^2_{com}(]s,T[,H_+)$, for all $T> t_1 > s_1>s$,
\begin{align*}
\int_{s_1}^{t_1} \langle |b(t) \cdot \nabla u,(\lambda-\Delta)^{\frac{1}{2}}\varphi| \rangle dt  \leq \delta \|u\|_{L^2([s_1,t_1],H_+)}\|\varphi\|_{L^2([s_1,t_1],H_+)}.
\end{align*}
\end{proposition}
\begin{proof}
It suffices to consider $b=b_n$, where $b_n=\mathbf{1}_{|b| \leq n}b$, and then take $n \rightarrow \infty$ using Fatou's Lemma. Thus, without loss of generality, $b$ is bounded. We have, using H\"{o}lder's inequality,
\begin{align*}
\langle |b(t) \cdot \nabla u,(\lambda-\Delta)^{\frac{1}{2}}\varphi| \rangle & =\langle |b(t)^{\frac{1}{2}} \cdot (\lambda-\Delta)^{-\frac{1}{4}}(\lambda-\Delta)^{\frac{1}{4}}\nabla u|,\big||b(t)|^{\frac{1}{2}}(\lambda-\Delta)^{-\frac{1}{4}}(\lambda-\Delta)^{\frac{3}{4}}\varphi\big|\rangle \\
& \leq \||b|^{\frac{1}{2}}(\lambda-\Delta)^{-\frac{1}{4}}\|_{2 \rightarrow 2}\|(\lambda-\Delta)^{\frac{1}{4}}\nabla u\|_2 \||b|^{\frac{1}{2}}(\lambda-\Delta)^{-\frac{1}{4}}\|_{2 \rightarrow 2}\|(\lambda-\Delta)^{\frac{3}{4}}\varphi\|_2.
\end{align*} 
Note that
\begin{align*}
\|(\lambda-\Delta)^{\frac{1}{4}}\nabla u\|^2_2 &=\langle \nabla(\lambda-\Delta)^{\frac{1}{4}}u,\nabla (\lambda-\Delta)^{\frac{1}{4}}u\rangle \\
& \leq \|(\lambda-\Delta)^{\frac{3}{4}}u\|^2_2.
\end{align*}
The result now follows upon applying condition $b \in L^\infty\,\mathbf{F}_\delta^{\scriptsize 1/2}$.
\end{proof}

\begin{proposition}
\label{prop1}
Let $\delta<1$.
Let $u \in L^\infty_{\loc}(]s,T],H) \cap L^2_{\loc}(]s,T],H_+)$ be a weak solution to equation \eqref{eq0_}. Then the following is true:

{\rm(\textit{i})} $\partial_t u \in L^2_{\loc}(]s,T],H_-)$;

{\rm(\textit{ii})} $u \in C(]s,T],H)$ (after redefinition on a measure zero set);

{\rm(\textit{iii})} for all $T \geq t_1 \geq s_1>s$,
$$
\|u(t_1)\|_{H}^2 + 2(1-\delta)\int_{s_1}^{t_1} \|u(t)\|^2_{H_+}dt \leq  \|u(s_1)\|_{H}^2.
$$
\end{proposition}

We prove Proposition \ref{prop1} in Appendix \ref{lions_app} by specifying the results in \cite{LM} to tripe $H_+ \subset H \subset H_-$.

\medskip

Armed with Proposition \ref{prop1}, we now prove Theorem \ref{thm0}(\textit{i}) using  some standard arguments. We include the details for the sake of completeness, and also because some care need to be taken when discussing approximation involving non-local operators.

1.~First, let $f \in \mathcal S$.  Fix $b_n \in C^\infty(\mathbb R_+ \times \mathbb R^d) \cap L^\infty(\mathbb R_+ \times \mathbb R^d)$ that have the same weak form-bound $\delta$ and $\lambda$ as $b$  (so, independent of $n$), and converge to $b$ in $L^1_{\loc}(\mathbb R_+ \times \mathbb R^{d})$ (see Proposition \ref{fbd_approx_prop}). Let $u_n$ denote the strong solution to Cauchy problem
$$
(\partial_t+\lambda - \Delta + b_n \cdot \nabla)u_n=0, \quad u_n(s)=f.
$$
Multiplying the last equation by $(\lambda-\Delta)^{\frac{1}{2}}u_n$ and integrating, we obtain a uniform in $n$ bound
\begin{equation}
\label{un_bd}
\sup_{t \in [s,T]}\|u_n(t)\|^2_{H} + 2(1-\delta)\int_{s}^T \|u_n(t)\|^2_{H_+} dt \leq \|f\|_{H}^2.
\end{equation}
Using a weak compactness argument, we can find a subsequence (also denoted by $u_n$) and a function $u \in L^\infty ([s,T],H) \cap L^2([s,T],H_+)$ such that
\begin{equation}
\label{un_conv}
u_n \rightarrow u  \quad \text{ weakly in } L^2([s,T],L^2), L^2([s,T],H_+).
\end{equation}

\smallskip

2.~Next, let us show that $u$ is a weak solution to \eqref{eq0_}. Let $\varphi \in C_c^\infty(]s,T[,\mathcal S)$. Let us pass to the limit $n \rightarrow \infty$ in
 \begin{align*}
-\int_s^T \langle u_n,\partial_t \varphi \rangle_H dt + \int_s^T \langle u_n,\varphi \rangle_{H_+} dt  
+ \int_s^T \langle b_n(t) \cdot \nabla u_n,(\lambda-\Delta)^{\frac{1}{2}}\varphi\rangle=0.
\end{align*}
Only the last term requires a comment. We have
\begin{align*}
&\langle b(t)\cdot \nabla u - b_n(t)\cdot \nabla u_n,(\lambda-\Delta)^{\frac{1}{2}}\varphi\rangle\\
&=\langle b(t)\cdot (\nabla u - \nabla u_n),(\lambda-\Delta)^{\frac{1}{2}}\varphi\rangle + \langle (b(t)- b_n(t))\cdot \nabla u_n,(\lambda-\Delta)^{\frac{1}{2}}\varphi\rangle=I_1+I_2.
\end{align*}
Let us show that $\int_s^T I_1$, $\int_s^T I_2 \rightarrow 0$ as $n \rightarrow \infty$.

We have
$$
I_1=-\langle (\lambda-\Delta)^{\frac{3}{4}}(u-u_n),\nabla (\lambda-\Delta)^{-\frac{3}{4}}\cdot b(t)(\lambda-\Delta)^{\frac{1}{2}}\varphi \rangle,
$$
where, clearly,
$$
\|\nabla (\lambda-\Delta)^{-\frac{3}{4}}\cdot b(t)(\lambda-\Delta)^{\frac{1}{2}}\varphi\|_2 \leq \delta \|\varphi\|_{H_+},
$$
so $\nabla (\lambda-\Delta)^{-\frac{3}{4}}\cdot b(\lambda-\Delta)^{\frac{1}{2}}\varphi \in L^2([s,T],L^2)$, and hence $\int_s^T I_1 dt \rightarrow 0$ as $n \rightarrow \infty$. 

Next, 
$$
I_2=\langle \nabla (\lambda-\Delta)^{\frac{1}{4}}u_n,(\lambda-\Delta)^{-\frac{1}{4}}(b(t)-b_n(t))(\lambda-\Delta)^{\frac{1}{2}}\varphi \rangle.
$$
Here $\nabla (\lambda-\Delta)^{\frac{1}{4}}u_n$ is uniformly in $n$ bounded in $L^2([s,T],L^2)$, while
\begin{equation}
\label{I2}
\|(\lambda-\Delta)^{-\frac{1}{4}}(b(t)-b_n(t))(\lambda-\Delta)^{\frac{1}{2}}\varphi \|_2 \leq 2\sqrt{\delta}\||b(t)-b_n(t)|^{\frac{1}{2}}(\lambda-\Delta)^{\frac{1}{2}}\varphi\|_2.
\end{equation}
Thus, we are left to show that
\begin{equation}
\label{I3}
\lim_n\int_s^T\||b(r)-b_n(r)|^{\frac{1}{2}}(\lambda-\Delta)^{\frac{1}{2}}\varphi\|^2_2dr=0.
\end{equation}
Fix a smooth function $\sigma$ on $\mathbb R_+$ taking values in $[0,1]$,
such that  $\sigma(t)= 1$ on $[0,1]$ and $0$ on $[2,\infty[$, and put $\zeta(x)=\sigma(\frac{|x|}{R})$, $R>0$. Then $|\nabla \zeta(x)|\leq c_1R^{-1}\mathbf 1_{|\nabla\zeta|}$ and $|\Delta\zeta(x)|\leq c_2R^{-2}\mathbf 1_{|\nabla\zeta|}$, where $\mathbf 1_{|\nabla\zeta|}$ is the indicator of the set where $|\nabla\zeta| \neq 0$.
Put, for brevity, $h:=(\lambda-\Delta)^{\frac{1}{2}}\varphi$. We have
\begin{align*}
\lim_n\int_s^T \||b(r)-b_n(r)|^{\frac{1}{2}} h \|^2_2dr&\leq \lim_n(\int_s^T \||b(r)-b_n(r)|^{\frac{1}{2}}\zeta h \|^2_2dr+\int_s^T \||b(r)-b_n(r)|^{\frac{1}{2}}(1-\zeta) h \|^2_2dr)\\
&(\text{we use that } b_n \rightarrow b\text{ in } L^1_{\loc}(\mathbb R_+ \times \mathbb R^{d}))\\
&=\lim_n\int_s^T \||b(r)-b_n(r)|^{\frac{1}{2}}(1-\zeta) h \|^2_2dr,
\end{align*}
and
\begin{align*}
& \||b(t)-b_n(t)|^{\frac{1}{2}}(1-\zeta)h(t)\|_2  \leq 2\sqrt{\delta}\|(\lambda-\Delta)^{\frac{1}{4}}(1-\zeta)h(t)\|_2 \\
&=2\sqrt{\delta}\|(\lambda-\Delta)^{-\frac{3}{4}}(\lambda-\Delta)(1-\zeta)h(t)\|_2 \\
& \leq 2\sqrt{\delta}\big(\lambda^\frac{1}{4}\|(1-\zeta)h(t)\|_2 + \lambda^{-\frac{3}{4}}\|\Delta[(1-\zeta)h(t)]\|_2\big)\\
&=o(R) \text{ as $R \rightarrow \infty$ due to the choice of } \zeta.
\end{align*}
Thus \eqref{I3} is proved.

 We obtain $\int_s^T I_2 \rightarrow 0$ as $n \rightarrow \infty$.
Hence
 \begin{align*}
& -\int_s^T \langle u,\partial_t \varphi \rangle_H dt  + \int_s^T \langle u,\varphi \rangle_{H_+} dt  
+ \int_s^T\langle b(t) \cdot \nabla u,(\lambda-\Delta)^{\frac{1}{2}}\varphi\rangle=0
\end{align*}
for all $\varphi \in C_c^\infty(]s,T[,\mathcal S)$,
i.e.\,$u$ is a weak solution to \eqref{eq0_}. 

\smallskip

3. We show that $u$ satisfies the initial condition 
\begin{equation}
\label{u_ic}
\lim_{t \downarrow s}\langle u(t),\psi\rangle=\langle f,\psi \rangle \quad \text{for all $\psi \in H_-$}.
\end{equation}
 First, consider $\psi \in H$. Put
$
g_n(t):=\langle u_n(t),\psi\rangle$, $t \in [s,T].
$
Then, for all $t$, $t+\Delta t \in [s,T]$,  
\begin{align*}
& |g_n(t+\Delta t) - g_n(t)|  \leq \int_t^{t+\Delta t} |\langle (\lambda-\Delta)^{\frac{3}{4}}u_n,(\lambda-\Delta)^{\frac{1}{4}}\psi\rangle |dr + \int_t^{t+\Delta t} |\langle b_n \cdot \nabla u_n, \psi\rangle | dr \\
& \leq \|u_n\|_{L^2([t,t+\Delta t],H_+)} \|\psi\|_{L^2([t,t+\Delta t],H)} + c(d)\delta \|u_n\|_{L^2([t,t+\Delta t],H_+)} \|\psi\|_{L^2([t,t+\Delta t],H)},
\end{align*}
where, estimating the last term, we argued as in the proof of Proposition \ref{lem_fbd_est}. 

Also, $\|u_n\|_{L^\infty([s,T] \times \mathbb R^d)} \leq \|f\|_\infty$, so we can apply the Arzel\`{a}-Ascoli Theorem. The latter, combined with \eqref{un_conv}, allows to refine the subsequence $\{u_n\}$ found earlier to obtain convergence of continuous functions
$$
\langle u_n(t),\psi\rangle \rightarrow \langle u(t),\psi\rangle \quad \text{ uniformly on }[T,s] - \Omega_\psi,
$$
where $\Omega_\psi$ is a measure zero set. Moreover, using the separability of $H$ and the bound
\begin{equation}
\label{bb}
\sup_{t \in [s,T]}\|u(t)\|_{H} \leq C\|f\|_{H}, 
\end{equation}
and applying the diagonal argument, we can further refine $\{u_n\}$ to obtain the uniform convergence on $[T,s] - \Omega$ where $\Omega$ is a measure zero set independent of $\psi$.
Further, since $u$ is a weak solution to \eqref{eq0_} and hence is in $C(]s,T],H)$ by Proposition \ref{lem_fbd_est}, we obtain that $t \mapsto \langle u(t),\psi\rangle$ can be uniquely extended to a continuous function on $[s,T]$ which must coincide at the endpoint $t=s$ with $\langle f,\psi\rangle$, i.e.\,we have \eqref{u_ic} for $\psi \in H$. Finally, using
$$
|\langle u(t),\psi \rangle| \leq \|u(t)\|_{H}\|\psi\|_{H_-}
$$
and the fact that $H$ is dense in $H_-$,
we obtain \eqref{u_ic} for all $\psi \in H_-$.

\smallskip

4.\,Given a weak solution $u$ to Cauchy problem \eqref{eq0_}, \eqref{cp}, we note that Proposition \ref{prop1}(\textit{iii}) gives
\begin{equation}
\label{ei2_}
\|u(t_1)\|_{H}^2 + 2(1-\delta)\int_{s}^{t_1} \|u(t)\|^2_{H_+}dt \leq  \|f\|_{H}^2.
\end{equation}
Indeed, the initial condition $u(t) \rightarrow f$ weakly in $H$  as $t \downarrow s$ gives $\|f\|_{H} \leq \liminf_{t \downarrow s}\|u(t)\|_{H}$.
At the same time,  Proposition \ref{prop1}(\textit{iii}) yields that $t \mapsto \|u(t)\|_{H}$ is a non-increasing function, so $$\|f\|_{H}= \lim_{t \downarrow s}\|u(t)\|_{H}$$ Hence $u(t) \rightarrow f$ (strongly) in $H$ as $t \downarrow s$, which yields \eqref{ei2_}.
Combined with Proposition \ref{prop1}(\textit{ii}), this gives
$
u \in C([s,T],H).
$
The uniqueness of weak solution follows from \eqref{ei2_}. The reproduction property of the evolution family $T^{t,s}f:=u(t)$ is a consequence of the uniqueness. The fact that this evolution family is positivity preserving $L^\infty$ contraction (and hence is Markov) is immediate from the construction of the weak solution $u$ via approximation, cf.\,\eqref{un_conv}. The convergence result follows from the weak compactness argument carried out above and the uniqueness of weak solution. Finally, the energy inequality \eqref{ei2_} and the fact that $\mathcal S$ is dense in $H$ allow to extend these results to an arbitrary $f \in H$.  

\medskip

To end the proof of Theorem \ref{thm0}(\textit{i}), we need to address the question of existence of a bounded smooth approximation of vector field $b$ preserving its weak form-bound $\delta$ and constant $\lambda=\lambda_\delta$. We put $b=0$ for $t<0$.

\begin{proposition}
\label{fbd_approx_prop}
Set $\bar{b}_\varepsilon(t):=e^{\gamma_\varepsilon(t)\Delta}\mathbf{1}_\varepsilon b(t)$, $t \in \mathbb R,$ $$b_\varepsilon:=c_\varepsilon e^{\varepsilon \Delta_1}\bar{b}_\varepsilon,$$ 
where $c_\varepsilon >0$, $\Delta_1:=\frac{\partial^2}{\partial t^2}$, $\gamma_\varepsilon $ is a $]0,1]$-valued measurable function on $\mathbb R_+$, $b$ and $\gamma_\varepsilon$ are extended by $0$ to $t<0$, $\mathbf{1}_\varepsilon$ is the indicator function of $\{(t,x) \mid t \in [0,\varepsilon^{-\frac{1}{2}}], |x| \leq \varepsilon^{-1}, |b(t,x)| \leq \varepsilon^{-1}\}$. 
There exist $c_\varepsilon \uparrow 1$, $\gamma_\varepsilon(t) \downarrow 0$ for every $t \in \mathbb R_+$  as $\varepsilon \downarrow 0$ such that
$$
b_\varepsilon \in L^\infty\,\mathbf{F}_\delta^{\scriptscriptstyle 1/2} \quad \text{ with $\lambda$ independent of $\varepsilon$}, 
$$
and
$$
b_\varepsilon \rightarrow b \quad \text{ in } L^1_{\loc}(\mathbb R_+ \times \mathbb R^d).
$$
\end{proposition}
\begin{proof}
First, we note that the convergence $\bar{b}_\varepsilon \rightarrow b$ in $L^1_{\loc}(\mathbb R_+ \times \mathbb R^d)$ is straightforward, provided that $\gamma_\varepsilon(t) \downarrow 0$ sufficiently rapidly as $\varepsilon \downarrow 0$.
Let us show that for any $\delta_\varepsilon \downarrow \delta$ we can select $\gamma_\varepsilon \downarrow 0$ fast enough so that
$$
\bar{b}_\varepsilon \in L^\infty\,\mathbf{F}_{\delta_\varepsilon}^{\scriptscriptstyle 1/2} \quad \text{ with the same $\lambda$}. 
$$
We have $$\bar{b}_\varepsilon=\mathbf{1}_\varepsilon b + (\bar{b}_\varepsilon-\mathbf{1}_\varepsilon b),$$ where, 
clearly, $\||\mathbf{1}_\varepsilon b|^{\frac{1}{2}}(t)(\lambda-\Delta)^{-\frac{1}{4}}\|_{2 \rightarrow 2} \leq \sqrt{\delta}$ for a.e.\,$t \in \mathbb R_+$, for all $\varepsilon$, while 
$\bar{b}_\varepsilon-\mathbf{1}_\varepsilon b \in L^\infty(\mathbb R_+,L^d)$. 
It follows from H\"{o}lder's inequality and the Sobolev Embedding Theorem that
for any $g \in L^2$, for a.e.\,$t \in \mathbb R_+$,
\begin{align*}
\||\bar{b}_\varepsilon(t)-\mathbf{1}_\varepsilon b(t)|^{\frac{1}{2}}(\lambda-\Delta)^{-\frac{1}{4}}g\|_2 & \leq \||\bar{b}_\varepsilon(t)-\mathbf{1}_\varepsilon b(t)|^{\frac{1}{2}}\|_{2d} \|(\lambda-\Delta)^{-\frac{1}{4}}g\|_{\frac{2d}{d-1}} \\
&\leq c\|\bar{b}_\varepsilon(t)-\mathbf{1}_\varepsilon b(t)\|^{\frac{1}{2}}_d \|g\|_2.
\end{align*}
For every $\varepsilon>0$ and every $t \in \mathbb R_+$, we can select $\gamma_\varepsilon(t)$ sufficiently small so that 
$\|\bar{b}_\varepsilon(t)-\mathbf{1}_\varepsilon b(t)\|^{\frac{1}{2}}_d  \leq c^{-1}\sqrt{\delta_\varepsilon-\delta}$.
Thus
$
\sup_{t \in \mathbb R_+}\||\bar{b}_\varepsilon(t)-\mathbf{1}_\varepsilon b(t)|^{\frac{1}{2}}(\lambda-\Delta)^{-\frac{1}{4}}g\|_{2} \leq \sqrt{\delta_\varepsilon-\delta}\|g\|_2.
$
It follows that for a.e.\,$t \in \mathbb R_+$, $\langle |\bar{b}_\varepsilon(t)|[(\lambda-\Delta)^{-\frac{1}{4}}g]^2\rangle \leq \delta_\varepsilon\|g\|^2_2$, $g \in L^2$, and hence
$$
\||\bar{b}_\varepsilon (t)|^{\frac{1}{2}}(\lambda-\Delta)^{-\frac{1}{4}}\|_{2 \rightarrow 2} \leq \sqrt{\delta_\varepsilon}.
$$

 Finally, recalling that $b_\varepsilon=c_\varepsilon \bar{b}_\varepsilon$, it is clear now that we can take $c_\varepsilon:=\frac{\delta}{\delta_\varepsilon}$ obtaining $$\||b_\varepsilon|^{\frac{1}{2}}(t)(\lambda-\Delta)^{-\frac{1}{4}}\|_{2 \rightarrow 2} \leq \sqrt{\delta}.$$
\end{proof}

\section{Proof of Theorem \ref{thm0_}}

Set $A=-\nabla \cdot a  \cdot \nabla$.
We prove Theorem \ref{thm0_} by first establishing an upper Gaussian bound on the heat kernel of the auxiliary operator
$$
H^+ = A  + b  \cdot \nabla + \mydiv b_+.
$$
Let $H^{t,s}f$ denote the solution of
\[ 
	\left\{ \begin{array}{rcl}
		-\frac{d}{d t} H^{t,s} f = H^+ H^{t,s} f & , & 0 \leq s < t <\infty \\
	0 \leq f \in L^1 \cap L^\infty & 
	\end{array} \right.
\]
in $L^p =L^p(\mathbb{R}^d), \;p \in [1,\infty[.$ 
Let $h(t,x;s,y)$ denote the heat kernel of $H^+$, that is, $$H^{t,s}f=\langle h(t,x;s,\cdot)f(\cdot)\rangle.$$

\begin{theorem}
\label{aux_ub_thm}
There exist generic constants $c_3,c_4 > 0$, $\omega\geq 0$ such that
\[
h(t,x;s,y) \leq  c_3 k_{c_4}(t-s,x-y)
\tag{$\mbox{UGB}^{h_+}$}
\]
for all $0 \leq s<t <\infty$.
\end{theorem}

\begin{proof}[Proof of Theorem \ref{aux_ub_thm}] 
We follow \cite{FS}.
We consider
\[ 
	\left\{ \begin{array}{rcl}
		-\frac{d}{d t} H_\alpha^{t,s} f = H_\alpha^+ H_\alpha^{t,s} f & , & 0 \leq s < t <\infty, \\
	0 \leq f \in L^1 \cap L^\infty & 
	\end{array} \right.
\]
in $L^p =L^p(\mathbb{R}^d), \;p \in [1,\infty[,$ where $H^{t,s}_\alpha := e^{\alpha \cdot x} H^{t,s} e^{-\alpha \cdot x}$ and 
\begin{align*}
 H_\alpha^+ := e^{\alpha \cdot x} H^+ e^{-\alpha \cdot x} & = H^+ - \alpha \cdot b - \alpha \cdot a \cdot \alpha  + \alpha \cdot a \cdot \nabla + \nabla \cdot a \cdot \alpha.
\end{align*}

\begin{lemma} 
\label{lem_n}
There are generic constants $c, c_4$ such that, for all $0 \leq s < t < \infty,$
\[
\|H^{t,s}_\alpha \|_{2 \to \infty}, \|H^{t,s}_\alpha \|_{1 \to 2} \leq c (t-s)^{-d/4} e^{c_4 \alpha^2(t-s)}.
\]
\end{lemma}

\begin{proof}[Proof of Lemma] 
Set $u_\alpha(t) := H_\alpha^{t,s} f$, $f=\Real f\in C_c^\infty$, $v(t):= u_\alpha^{p/2}(t)$, $p\geq 2.$ Noticing that $$\langle b \cdot \nabla u_\alpha, u_\alpha^{p-1} \rangle = \frac{2}{p} \langle \nabla v, b v \rangle = -\frac{1}{p} \langle v^2, \mydiv b \rangle,$$ we have, using the equation,
\begin{align*}
-\frac{1}{p} \frac{d}{d t} \langle v^2(t) \rangle & = \frac{4}{p p^\prime} \|A^{1/2} v(t) \|_2^2 +\frac{1}{p^\prime} \langle v^2(t),\mydiv b_+ \rangle +\frac{1}{p} \langle v^2(t), \mydiv b_- \rangle \\
& - \frac{2(p-2)}{p} \langle \alpha \cdot a \cdot \nabla v(t), v(t) \rangle  -\langle \alpha \cdot b, v^2(t) \rangle -\langle \alpha \cdot a \cdot \alpha, v^2(t) \rangle.
\end{align*}
By $b \in \mathbf{M}_\delta$,
\begin{align*}
|\langle\alpha \cdot  b,v^2\rangle| \leq |\alpha||\langle b,v^2\rangle|&\leq |\alpha|\delta\|\nabla v\|_2\|v\|_2 +|\alpha| g \|v\|_2^2\\
&\leq \frac{4\gamma}{pp'}\|A^\frac{1}{2}v\|_2^2+\bigg(|\alpha| g+\frac{1}{4\gamma}\frac{pp'}{4}
\frac{\delta^2}{\sigma}|\alpha|^2\bigg)\|v\|_2^2  \quad (\gamma>0),
\end{align*}
 and so, applying the quadratic inequality, we have
 \[
- \frac{d}{d t}\|v\|_2^2 \geq \frac{4}{p'}(1-2\gamma)\|A^{1/2} v \|_2^2  -\bigg(p|\alpha| g+ \frac{p^2p'}{4\gamma}\bigg[\frac{4\gamma}{pp'}\xi+\bigg(\frac{p-2}{p}\bigg)^2\xi+\frac{\delta^2}{4\sigma}\bigg]\alpha^2\bigg)\|v\|_2^2
\]

Putting $p=2$ and $\gamma=\frac{1}{2}$, we have, taking into our assumption on $g$, i.e.\,$\int_s^t g ds \leq c_\delta\sqrt{t-s}$,
\[
\|u_\alpha(t) \|_2 \leq \|f\|_2 \exp\bigg[\bigg(\xi+\frac{\delta^2}{2\sigma}\bigg)\alpha^2(t-s)+c_\delta|\alpha|\sqrt{t-s}\bigg] 
\tag{$\star^a$}
\]

Also, using Nash's inequality
\begin{equation}
\label{nash_ineq}
\|\nabla \psi \|_2^2 \geq C_N \| \psi \|_2^{2 +4/d} \|\psi\|_1^{-4/d}, \quad \psi \in W^{1,2} \cap L^1,
\end{equation}
we have, putting $\gamma=\frac{1}{4}$, $p\geq 4$ and setting $F(t,\alpha)=p|\alpha| g(t)+p^2\big(\xi +\frac{\delta^2}{3\sigma}\big)\alpha^2$, 
\begin{align*}
 -\frac{d}{d t}\|v\|_2^2 & \geq \frac{2}{p'}\|A^{1/2}v\|_2^2 -F(t,\alpha) \|v \|_2^2 \\
& \geq \frac{2}{p'}\sigma C_N \|v\|_2^{2+\frac{4}{d}} \|v\|_1^{-\frac{4}{d}} -F(t,\alpha) \|v \|_2^2,
\end{align*}
so
\begin{align*}
\frac{d}{d t} \|v\|_2^{-4/d} & \geq \frac{4\sigma C_N}{dp'} \|v\|_1^{-\frac{4}{d}} - \frac{2}{d} F(t,\alpha)\|v \|_2^{-4/d}.
 \end{align*}

 Setting $c_g=\frac{4\sigma C_N}{dp'}, \;w(t)=\|v\|_2^{-\frac{4}{d}}$ and $\mu(t)=\frac{2}{d}\int_s^tF(r,\alpha)d r ,$ 
we have
\begin{align*}
w_p(t) & \geq c_g e^{-\mu_p(t)} \int_s^t e^{\mu_p(r)} w_\frac{p}{2} (r)d r \\
& \geq  c_g e^{-\mu_p(t)} \;V_\frac{p}{2}(t) \int_s^t e^{\mu_p(r)} (r-s)^q d r ,
\end{align*}
where $ q = \frac{p}{2} - 2$ and $V_\frac{p}{2}(t):= \bigg \{ \sup \bigg[(r-s)^\frac{q d}{2 p} \|u_\alpha(r)\|_{p/2} \mid s\leq r \leq t\bigg] \bigg \}^{-\frac{2p}{d}}.$

Set $\kappa = \frac{2}{d} \bigg(\frac{3}{2}\xi+\frac{\delta^2}{3\sigma}\bigg) \alpha^2.$ Since for $s \leq r \leq t,$
\begin{align*}
-\mu_p(t) + \mu_p(r) & = - \frac{2}{d}\bigg[p^2\big(\xi +\frac{\delta^2}{3\sigma}\big)\alpha^2(t-r) + p |\alpha|\int_r^t g(\tau) d \tau \bigg] \\
& \geq - \frac{2}{d}\bigg[p^2\big( \xi+ \frac{\delta^2}{3\sigma}\big)\alpha^2(t-r) + p |\alpha| c_\delta \sqrt{t-r}  \bigg]\\
& \geq -\frac{2}{d}\bigg[p^2 \bigg(\frac{3}{2} \xi+ \frac{\delta^2}{3\sigma}\bigg) \alpha^2(t-r) + \frac{c_\delta^2}{2\xi} \bigg]\\
& = -\bigg[\kappa p^2(t-s) +\frac{c_\delta^2}{d\xi} \bigg] + \kappa p^2 (r-s), \text{ and so }\\
e^{-\mu_p(t)} \int_s^t e^{\mu_p(r)} (r-s)^q d r & \geq e^{-\kappa p^2 (t-s) - \frac{c_\delta^2}{d\xi}} \int_s^t e^{\kappa  p^2 (r-s)}(r-s)^q d r \text{ and } \\
\int_s^t e^{\kappa  p^2 (r-s)}(r-s)^q d r & \geq K p^{-2} (t-s)^\frac{p-2}{2} e^{\kappa (p^2 -1)(t-s)},
\end{align*}
where $K:= 2 \inf \big \{ p \big[1-(1-p^{-2})^{p-1} \big] \mid p \geq 2 \big \} > 0,$ we obtain
\[
w_p(t) \geq c_g K p^{-2} e^{-\kappa (t-s) - \frac{ c_\delta^2}{d\xi}} (t-s)^\frac{p-2}{2} V_\frac{p}{2}(t),
\]
or, setting $W_p(t) := \sup \big[ (r-s)^\frac{d(p-2)}{4 p} \|u_\alpha (r) \|_p \mid s \leq r \leq t \big],$
\[
W_p(t) \leq (c_g K e^{-\frac{2 c_\delta^2}{d\xi}})^{-\frac{d}{2 p}} p^\frac{d}{p} e^{(\frac{3}{2}\xi+\frac{\delta^2}{3\sigma})\frac{\alpha^2}{p} (t-s)} W_{p/2}(t), \;\; p = 2^k, \;k= 1, 2, \dots.
\]
Iterating this inequality, starting with $ k = 2,$ yields $$(t-s)^\frac{d}{4} \|u_\alpha (t)\|_\infty \leq C_g  e^{(\frac{3}{2}\xi+\frac{\delta^2}{3\sigma}) \alpha^2 (t-s)} W_2 (t).$$ Finally, taking into account $(\star^a),$ we arrive at $$\|H^{t,s}_\alpha \|_{2 \to \infty} \leq (t-s)^{-d/4} C_g e^{(3\xi+\frac{\delta^2}{\sigma}) \alpha^2 (t-s)}.$$ The same bound holds for $\|H^{t,s}_\alpha(b) \|_{1 \to 2}=\|\big(H^{t,s}_{-\alpha}(-b)\big)^*\|_{2\to \infty}$. This ends the proof of Lemma \ref{lem_n}.
\end{proof}

From Lemma \ref{lem_n} we obtain $$h(t,x;s,y)\leq C_g^2 (t-s)^{-\frac{d}{2}}e^{\alpha\cdot(y-x)+c_4\alpha^2(t-s)}, \quad c_4=3\xi+2\frac{\delta^2}{3\sigma}.$$ Putting $\alpha=\frac{x-y}{2c_4(t-s)}$, we obtain 
$(\mbox{UGB}^{h_+})$. This completes the proof of Theorem \ref{aux_ub_thm}.
\end{proof}

We are in position to complete the proof of Theorem \ref{thm0_}.
We consider operator $A + b \cdot \nabla$ as $H^+$ perturbed by potential $- \mydiv b_+$. Hence, the sought upper bound on the heat kernel of $A + b \cdot \nabla$ follows from Theorem \ref{aux_ub_thm} and a standard argument based on the Duhamel formula using that the Kato norm $\mu_+$ of $({\rm div\,} b)_+$ is sufficiently small. \qed

\bigskip

\section{A priori lower bound}

\label{apr_lb_thm}

In this section,
$u(t,x;s,y)$ denotes the heat kernel of operator $-\nabla \cdot a \cdot \nabla + b \cdot \nabla$
with matrix $a \in H_{\sigma,\xi}$, vector field $b$ and ${\rm div\,}b$ assumed to be bounded $C^\infty$ smooth.

\begin{theorem}[a priori lower bound] 
\label{thm2}

In the assumptions of Theorem \ref{thm0__}, a global in time lower Gaussian  bound
\[
\tag{${\rm LGB}$}
\label{lgb_apr}
c_1 k_{c_2}(t-s;x-y)\leq u(t,x;s,y)
\]
holds for all $0 \leq s<t<\infty$, $x,y \in \mathbb R^d$ with generic$\ast$ constants $c_i$, $i=1,2$ that can also depend on the Kato norms $\mu_\pm$.
\end{theorem}

Throughout the rest of the proof, the constants that we find are generic$\ast$ that can also depend on $\mu_\pm$.

\medskip

The proof of Theorem \ref{thm2} (given in the end of this section) is based on the following estimates of Nash's $G$-functions.

\label{lb_sect}

\subsection{$\hat{G}$-function for $-\nabla \cdot a \cdot \nabla + (\hat{b}+\tilde{b})\cdot \nabla$}

Since $\hat{b}+\tilde{b}$ is in $\mathbf{M}_\delta$, Theorem \ref{thm0_} applies and gives 
\begin{equation}
\label{ub_11}
u(t,x;s,y) \leq \hat{c}_3 k_{c_4}(t-s;x-y)
\end{equation}
where, recall, $u(t,x;s,y)$ is the heat kernel $\Lambda=-\nabla \cdot a \cdot \nabla + (\hat{b} + \tilde{b}) \cdot \nabla$.
The constants in the next proposition depend on the same parameters as the constants in the theorem except for the Kato norms $\mu_{\pm}$ and ${\rm div\,}b_{\pm}$.

Set $$\tilde{Q}(t-s) := \frac{d}{2} \log (t-s).$$

\begin{proposition}
\label{g1_thm_3}
Let $x,y \in \mathbb R^d$. Put $o=\frac{x+y}{2}$, $t_s=\frac{t+s}{2}$.
There exist constants $\beta$ and $\mathbb{C}$ such that 
\begin{align*}
G(t_s) :=\langle k_\beta(t-t_s,o-\cdot) \log u(t,z;t_s,\cdot) \rangle 
 \geq - \tilde{Q}(t-t_s) - \mathbb{C}
\end{align*}
for all $z \in B(o,\sqrt{t-t_s})$.
\end{proposition}

\begin{proof}[Proof of Proposition \ref{g1_thm_3}]

Fix $\varepsilon>0$ and define
\[ 
G_\varepsilon(\tau):=\langle k_\beta(t,o;t_s,\cdot) \log \big[\varepsilon k_\beta(t,o;t_s,\cdot)  +u(t,z;\tau,\cdot) \big] \rangle.
\]
Then
\[
G(\tau)=\inf_{\varepsilon>0} G_\varepsilon (\tau), \;\; \tau \in \big [t_s, \frac{t+t_s}{2} \big] \text{ and  } G(t_s)=\Theta_\beta (t,t_s,z).
\]
Below we are using the following shorthand:
\[
G_\varepsilon(\tau) \equiv \big \langle \Gamma \log \big[ \varepsilon\Gamma + U \big] \big \rangle \equiv \big \langle \Gamma \log \big[ \varepsilon\Gamma + U(\tau) \big] \big \rangle,
\]
where $\Gamma \equiv \Gamma_\beta \equiv k_\beta(t,o;t_s,\cdot), \; U \equiv U(\tau) \equiv u(t,z;\tau,\cdot).$

Also set
\[
V := c_0 (t-t_s)^{d/2} \big[\varepsilon \Gamma+ U \big], \; c_0 = (4 \pi c_4)^{d/2} e^{-1}\big[\varepsilon + c_3 e^{\frac{1}{4c_4}}\big]^{-1},
\]
where $c_3, c_4$ are the constant from the upper Gaussian bound for $U.$

If $\beta \geq 2c_4,$ then clearly  
\begin{equation}
\label{V_est}
V(\tau,y) \exp \frac{|o-y|^2}{4 \beta (t-t_s)} \leq e^{-1} \text{ for all } y \in \mathbb{R}^d, \; \varepsilon \in ]0, 1] \text{ and } \tau \in \big [t_s,\frac{t+t_s}{2} \big].
\end{equation}

Let us calculate $- \partial_\tau G_\varepsilon(\tau)$:
\begin{align*}
& - \partial_\tau G_\varepsilon(\tau)  = \bigg \langle \Gamma \frac{-\partial_\tau U}{\varepsilon\Gamma + U} \bigg \rangle = \bigg \langle \frac{\Gamma}{\varepsilon \Gamma + U} (\nabla \cdot a \cdot \nabla + \nabla \cdot (\hat{b}+\tilde{b}))U \bigg \rangle \\
& = \bigg \langle \nabla \log V \cdot a \Gamma \cdot \frac{\nabla U}{\varepsilon \Gamma+ U} \bigg \rangle - \bigg \langle \nabla \Gamma \cdot a\cdot \frac{\nabla U}{\varepsilon\Gamma + U}  \bigg \rangle + \bigg \langle \Gamma \frac{(\hat{b}+\tilde{b}) \cdot\nabla U}{\varepsilon\Gamma + U} \bigg \rangle + \bigg \langle \Gamma \frac{{\rm div}\hat{b}\;U}{\varepsilon\Gamma + U} \bigg \rangle \\
& = \big \langle \nabla \log V \cdot a \Gamma \cdot \nabla \log V \big \rangle -\bigg \langle \nabla \log V \cdot a \Gamma \cdot \frac{\varepsilon\nabla \Gamma}{\varepsilon \Gamma+ U} \bigg \rangle - \big \langle \nabla \Gamma \cdot a \cdot \nabla \log V \big \rangle \\
&+\bigg \langle \nabla \Gamma \cdot a \cdot \frac{\varepsilon\nabla \Gamma}{\varepsilon\Gamma + U}  \bigg \rangle + \big \langle \Gamma (\hat{b}+\tilde{b}) \cdot \nabla \log V \big \rangle - \bigg \langle \Gamma \frac{(\hat{b}+\tilde{b}) \cdot\varepsilon\nabla \Gamma}{\varepsilon\Gamma + U} \bigg \rangle + \bigg \langle \Gamma \frac{U{\rm div}\hat{b}}{\varepsilon\Gamma + U} \bigg \rangle.
\end{align*}
Setting $$\mathcal{N}(\tau):= \big \langle \nabla \log V \cdot a \Gamma \cdot \nabla \log V \big \rangle,$$ applying the quadratic inequality and estimating $\big \langle \Gamma \frac{U{\rm div}\hat{b}}{\varepsilon\Gamma + U} \big \rangle \geq -\big \langle \Gamma{\rm div}\hat{b}_-\big \rangle,$ we have
\begin{align*}
- \partial_\tau G_\varepsilon(\tau) &\geq \mathcal{N} - 2 \mathcal{N}^{1/2} \bigg \langle \nabla \Gamma \cdot \frac{a}{\Gamma} \cdot \nabla \Gamma \bigg \rangle^{1/2} + J\\
& \geq (1-\gamma)\mathcal{N}-\frac{\xi}{\gamma}\bigg\langle\frac{(\nabla\Gamma)^2}{\Gamma}\bigg\rangle + J \quad (0<\gamma<1), 
\end{align*}
where
\begin{align*}
J : =  \big \langle \Gamma \tilde{b} \cdot \nabla \log V \big \rangle-\bigg \langle \Gamma \frac{\tilde{b} \cdot \varepsilon\nabla \Gamma}{\varepsilon\Gamma + U} \bigg \rangle  + \big \langle \Gamma \hat{b} \cdot \nabla \log V \big \rangle- \bigg \langle \Gamma \frac{\hat{b} \cdot \varepsilon\nabla \Gamma}{\varepsilon\Gamma + U} \bigg \rangle-\big \langle \Gamma{\rm div}\hat{b}_-\big \rangle.
\end{align*}
Applying $\langle \frac{|\nabla \Gamma|^2}{\Gamma}\rangle = \xi \frac{d}{2\beta(t-t_s)}$, we arrive at 
\begin{align}
\label{G_deriv_est}
- \partial_\tau G_\varepsilon(\tau) \geq (1-\gamma)\mathcal{N}(\tau) - \frac{\xi}{\gamma} \frac{d}{2\beta(t-t_s)} + J.
\end{align}
Define
$$
Y(\tau):=G_\varepsilon(\tau) + \tilde{Q}(t-\tau).
$$
Our goal is to show that $Y$ is bounded from below by a constant.
Note that \eqref{G_deriv_est} can be rewritten as
\begin{equation}
\label{Y_deriv}
-\partial_\tau Y(\tau) \geq (1-\gamma)\mathcal N(\tau) -\frac{\xi}{\gamma}\frac{d}{2\beta(t-t_s)} + J.
\end{equation}
Here we have used that $-\partial_\tau \tilde{Q}(t-\tau)=\frac{d}{2(t-\tau)}>0$.
Multiplying \eqref{Y_deriv} by $e^{\mu(\tau)}$, where
$$
\mu(\tau):= - \frac{t-\tau}{t-t_s} - \frac{F(\tau)}{\sqrt{\beta (t-t_s)}}-P(\tau) - H(\tau)
$$
where $F(\tau):=\sqrt{2} \int_\tau^t g(r) d r$, $P(\tau):=\int_\tau^t g^2(r)dr$, $H(\tau):=\int_\tau^t h(r)dr$, we obtain
$$
-\partial_\tau \bigl[e^{\mu(\tau)}Y(\tau)\bigr] \geq e^{\mu(\tau)}\bigl[ - Y(\tau)\partial_\tau \mu(\tau) +  (1-\gamma)\mathcal N(\tau) -\frac{\xi}{\gamma}\frac{d}{2\beta(t-t_s)} + J \bigr].
$$
Note that, due to our assumptions on $g$, $g^2$ and $h$, the function $\mu$ is uniformly bounded in variables $\tau \in [t_s,(t+t_s)/2]$ and $0 \leq s<t<\infty$.

\begin{lemma}
\label{cl1}
Let $c>\log (1+c_3)$ with $c_3$ from  $u(t,x;\tau,\cdot) \leq c_3 (t-\tau)^{-d/2}.$ Then, for all $\varepsilon>0$ sufficiently small,
$$Y(\tau) < c$$
for all $\tau \in [t_s, \frac{t+t_s}{2}]$.
\end{lemma}
\begin{proof}
Indeed, for $\varepsilon\leq (4\pi \beta)^\frac{d}{2}$ sufficiently small,
\begin{align*}
G_\varepsilon(\tau) & = \big \langle \Gamma \log (\varepsilon \Gamma + U) \big \rangle \\
& \leq \big \langle \Gamma \big \rangle \log \big[(1+c_3) (t-\tau)^{-d/2}] \\
& < -\tilde{Q}(t-\tau) + \log (1+c_3).
\end{align*}
\end{proof}

From now on, let $c$ and $\varepsilon$ be as in Lemma \ref{cl1}.
Then we obtain from \eqref{Y_deriv}
\begin{equation}
\partial_\tau \big(e^{\mu(\tau)} (Y(\tau) - c) \big)^{-1}  \geq \big[\gamma \mathcal{N}(\tau)  +\mathcal{M}(\tau)\big]e^{-\mu(\tau)} (Y(\tau)-c)^{-2},
\label{ast2}
\end{equation}
(division by zero is ruled out by Lemma \ref{cl1}),
where
\begin{align*}
\mathcal{M}(\tau) & :=- ( Y(\tau) - c) \partial_\tau \mu(\tau)+  (1-2\gamma) \mathcal{N}(\tau)  -\frac{\xi}{\gamma}\frac{d}{2\beta(t-t_s)} + J.
\end{align*}

\begin{lemma}
\label{cl2}
$\mathcal{M}(\tau) \geq 0$
for all $\tau \in \big[t_s, \frac{t+t_s}{2}\big]$, provided that $\nu_-<\sigma$, $\gamma$ is sufficiently small and $c$ is sufficiently large.
\end{lemma}

We prove Lemma \ref{cl2} below. 

Now, taking Lemma \ref{cl2} for granted, we complete the proof of  Proposition \ref{g1_thm_3}. Lemma \ref{cl2} and \eqref{ast2} give 
\begin{equation}
\partial_\tau \big(e^{\mu(\tau)} (Y(\tau) -c) \big)^{-1}  \geq \gamma\sigma \mathcal N_1(\tau) e^{-\mu(\tau)} (Y(\tau)-c)^{-2},
\label{ineq3}
\end{equation}
where, recall, $\mathcal N_1(\tau):=\langle \Gamma |\nabla \log V|^2\rangle$.
By the Spectral Gap Inequality (see e.g.\,\cite[Sect.\,2]{N}),
\begin{align*}
\mathcal{N}_1 & \geq \frac{1}{2 \beta (t-t_s)} \big \langle \Gamma |\log V -\langle \Gamma \log V \rangle |^2 \big \rangle \\
& = \frac{1}{2 \beta (t-t_s)} \big \langle \Gamma |\log \big[\varepsilon \Gamma + U \big] -\langle \Gamma \log \big[\varepsilon \Gamma + U \big] \rangle |^2 \big \rangle \\
& \equiv \frac{1}{2 \beta (t-t_s)} \big \langle \Gamma |\log \big[\varepsilon \Gamma + U \big] - G_\varepsilon |^2 \big \rangle.
\end{align*}
In turn,
$
\Gamma \geq CU
$ for a generic$\ast$ constant $C$.
(Indeed, $\frac{1}{2}|o-\cdot|^2\leq |z-\cdot|^2 + |o-z|^2$. Clearly, $\frac{1}{t-t_s} \leq \frac{1}{t-\tau} \leq \frac{2}{t-t_s}$ combined with $|z-o| \leq \sqrt{t-t_s}$ implies that $-\frac{|z-\cdot|^2}{4c_4(t-\tau)}\leq -\frac{|o-\cdot|^2}{8c_4(t-t_s)} + \frac{|o-z|^2}{2c_4(t-t_s)}$. Thus, if $\beta\geq 2c_4$, then 
\[
k_{c_4} (t,z;\tau, \cdot) \leq \big(\frac{2\beta}{c_4}\big)^\frac{d}{2} e^\frac{1}{2c_4} k_\beta (t,o;t_s,\cdot).
\]
Therefore, by $U \leq c_3 k_{c_4}(t,z;\tau, \cdot)$, see \eqref{ub_11}, and $\beta = 2c_4$, we have the required inequality $\Gamma \geq C U$ with $C^{-1}=c_3 2^d e^\frac{1}{\beta}$.)
Hence
\[
\mathcal{N}_1(\tau) \geq \frac{C}{2 \beta (t-t_s)} \big \langle U |\log \big[\varepsilon \Gamma + U \big] - G_\varepsilon |^2 \big \rangle,
\]
and so, by $\langle U \rangle = 1,$
\[
\mathcal{N}_1(\tau) \geq \frac{C}{2 \beta (t-t_s)} \big \langle U |\log \big[\varepsilon \Gamma + U \big] - G_\varepsilon | \big \rangle^2.
\]
Now,
\begin{align*}
\big \langle U |\log \big[\varepsilon \Gamma + U \big] - G_\varepsilon | \big \rangle & \geq \big \langle U \log \big[\varepsilon \Gamma + U \big] \big \rangle - G_\varepsilon \big \langle U \big \rangle \\
& \geq \big \langle U \log U \big \rangle - G_\varepsilon \big \langle U \big \rangle \\
& \geq - G_\varepsilon(\tau) - \tilde{Q}(t-\tau) - \mathcal{C} \\
& \equiv - Y(\tau) - \mathcal{C}.
\end{align*}
Here we again have used identity $\langle U \rangle = 1$ and the Nash entropy estimate $-\big \langle U \log U \big \rangle \leq \tilde{Q}(t-\tau) + \mathcal{C}.$
(The latter follows from  $e^\frac{Q}{d}\leq CM$, where $Q:=-\big \langle U \log U \big \rangle$, and $M(t,\tau)=\langle|z-\cdot|U\rangle\leq C\sqrt{t-\tau}$, which is a consequence of the upper bound $U \leq c_3 k_{c_4}(t,z;\tau, \cdot)$. The inequality $e^\frac{Q}{d}\leq CM$, in turn, follows from $\langle u \rangle = 1$ and $u \log u \geq - \alpha u -e^{-1-\alpha}$ for all real $\alpha.$)

\smallskip

\textit{Case} (a): For all $\tau \in [t_s, (t+t_s)/2]$, $-Y(\tau) -c -2 \mathcal{C} \geq 0$. Here $c$ is from \eqref{ineq3}. 

\smallskip

Then $- Y(\tau) - \mathcal{C} \geq \frac{1}{2} (-Y(\tau) +c ) >c+ \mathcal{C} > 0$ and hence
\[
\mathcal{N}_1(\tau) \geq  \frac{C}{8 \beta (t-t_s)} \big(-Y(\tau) + c \big)^2.
\]
Thus, by \eqref{ineq3}, $$\big( c - Y(t_s) \big)^{-1} \geq \frac{\gamma\sigma C}{8 (t-t_s)} e^{\mu(t_s)} \int_{t_s}^{(t+t_s)/2} e^{- \mu(\tau)} d \tau \geq \frac{\gamma\sigma C}{8 (t-t_s)} \int_{t_s}^{(t+t_s)/2} d \tau,$$ and so $c - Y(t_s) \leq \frac{16}{\gamma\sigma C } = \frac{2^{d +4} c_3}{\gamma\sigma}e^\frac{1}{\beta},$ or $$G_\varepsilon (t_s) \geq -\tilde{Q}(t-t_s) +c - \frac{2^{d +4} c_3}{\gamma\sigma}e^\frac{1}{\beta}.$$

\textit{Case} (b): For some $\tau \in [t_s, (t+t_s)/2]$, $-Y(\tau) -c -2 \mathcal{C} < 0$.

\smallskip

By \eqref{ineq3}, $$\big(e^{\mu(\tau)} (Y(\tau) -c) \big)^{-1} \geq \big(e^{\mu(t_s)} (Y(t_s) -c) \big)^{-1},$$ or $c -Y(t_s) \leq e^{\mu(\tau)-\mu(t_s)} (c - Y(\tau)).$ Therefore,
$c-Y(t_s) \leq e^{\mu(\tau)-\mu(t_s)} 2 (c + \mathcal{C}) \leq e^{-\mu(t_s)} 2 (c + \mathcal{C}),$ or $$G_\varepsilon (t_s) \geq -\tilde{Q}(t-t_s) +c - 2 (c + \mathcal{C})e^{1+\frac{d+2}{4\sqrt{\beta}}c_\delta+\|g\|^2_{L^2(\mathbb R_+)} + \|h\|_{L^1(\mathbb R_+)}}.$$

This ends the proof of  Proposition \ref{g1_thm_3}.
\end{proof}

\subsection{Proof of Lemma \ref{cl2}}

The main task is to estimate $J= \langle \Gamma \tilde{b} \cdot \nabla \log V  \rangle- \langle \Gamma \frac{\tilde{b} \cdot \varepsilon\nabla \Gamma}{\varepsilon\Gamma + U}  \rangle  - \langle \Gamma \frac{\hat{b} \cdot \varepsilon\nabla \Gamma}{\varepsilon\Gamma + U}  \rangle-\langle \Gamma({\rm div}\hat{b})_- \rangle + \langle \Gamma \hat{b} \cdot \nabla \log V  \rangle$ in the definition of $\mathcal M$.

1. Term ``$\langle \Gamma \tilde{b}\cdot\nabla\log V\rangle$''. There exists a bounded smooth skew-symmetric matrix $B \in L^\infty\,{\rm BMO}$ such that $\tilde{b}=\nabla B$. 
 We have
\begin{align*}
\langle \Gamma \tilde{b}\cdot\nabla\log V\rangle & =-\langle \tilde{b}\cdot \nabla\Gamma,\log V\rangle \\
& =-2\langle \tilde{b}\cdot \nabla\sqrt{\Gamma},\sqrt{\Gamma}\log V\rangle.
\end{align*}
Hence, by Proposition \ref{lem_bmo_est},
\begin{align*}
&|\langle \Gamma \tilde{b}\cdot\nabla\log V\rangle|\leq 2\|B\|_{\rm {BMO}}\|\nabla\sqrt{\Gamma}\|_2\|\nabla(\sqrt{\Gamma}\log V)\|_2,
\end{align*}
where
$$
\|\nabla\sqrt{\Gamma}\|_2=\sqrt{\frac{d}{8\beta}\frac{1}{t-t_s}},
$$
\begin{align*}
\|\nabla(\sqrt{\Gamma}\log V)\|_2^2 & =\big\langle \big((\nabla\sqrt{\Gamma})\log V  + \sqrt{\Gamma}\nabla\log V\big)^2\big\rangle \\
&=:\big\langle(p+q)^2\big\rangle=\big\langle p^2+2p\cdot q+q^2\big\rangle.
\end{align*}
Note that $\langle 2p\cdot q\rangle=\frac{1}{2}\langle\nabla\Gamma\cdot\nabla\log^2V\rangle=-\frac{1}{2}\langle \Delta\Gamma,\log^2 V\rangle$. Using the equality $\Delta\Gamma=\frac{(\nabla\Gamma)^2}{\Gamma}-\frac{d}{2\beta (t-t_s)}\Gamma$ we have $\langle 2p\cdot q\rangle=-2\langle p^2\rangle+\frac{d}{4\beta(t-t_s)}\langle\Gamma\log^2 V\rangle$. Thus, 
\begin{align*}
\|\nabla(\sqrt{\Gamma}\log V)\|_2 & \leq \|q\|_2+\sqrt{\frac{d}{4\beta (t-t_s)}}\|\sqrt{\Gamma}\log V\|_2 \\
& = \mathcal N^{\frac{1}{2}}_1(\tau) + \sqrt{\frac{d}{4\beta (t-t_s)}}\|\sqrt{\Gamma}\log V\|_2,
\end{align*}
where $\mathcal N_1(\tau):=\langle \Gamma |\nabla \log V|^2\rangle$ $(\leq \sigma^{-1}\mathcal N(\tau))$.
 Therefore, 
$$
|\langle \Gamma \tilde{b}\cdot\nabla\log V\rangle|\leq \frac{\gamma}{2}\mathcal N_1 + 
\|B\|_{\rm BMO}^2\frac{1}{\gamma}\frac{d}{4\beta(t-t_s)} + \|B\|_{\rm BMO}\frac{d}{2\sqrt{2}\beta(t-t_s)}\|\sqrt{\Gamma}\log V\|_2.
$$ 
Let us estimate the third term in the RHS of this inequality. We have
\begin{align*}
\frac{1}{t-t_s}\|\sqrt{\Gamma}\log V\|^2_2 & = \frac{1}{t-t_s}\langle \Gamma (\log V - \langle \Gamma \log V\rangle)^2 \rangle + \frac{1}{t-t_s}\langle \Gamma \log V\rangle^2 \\
& (\text{we are applying the Spectral Gap Inequality in the first term}) \\
& \leq 2\beta \mathcal N_1(\tau) + \frac{1}{t-t_s}\langle \Gamma \log V \rangle^2,
 \end{align*}
so
$$
\frac{1}{\sqrt{t-t_s}}\|\sqrt{\Gamma}\log V\|_2 \leq \sqrt{2\beta}\sqrt{\mathcal N_1} + \frac{\langle \Gamma (-\log V) \rangle}{\sqrt{t-t_s}} \qquad (\text{we use $-\log V \geq 1$, see \eqref{V_est}}).
$$
Setting $\varphi:=\Gamma (-\log V)$, we arrive at
\begin{align*}
|\langle \Gamma \tilde{b}\cdot\nabla\log V\rangle| \leq \gamma\mathcal N_1(\tau) +\|B\|_{\rm BMO}^2\frac{d^2}{8 \gamma} \frac{1}{2\beta(t-t_s)} +  \|B\|_{\rm BMO}\frac{d}{2\sqrt{2}\beta(t-t_s)}\langle \varphi \rangle.
\end{align*}

2.~Term ``$\langle \Gamma \frac{\tilde{b} \cdot \varepsilon\nabla \Gamma}{\varepsilon \Gamma + U} \rangle$''. By Proposition \ref{lem_bmo_est},
$$
|\bigg \langle \Gamma \frac{\tilde{b} \cdot \varepsilon\nabla \Gamma}{\varepsilon \Gamma + U} \bigg \rangle|= 2|\bigg \langle \tilde{b}\cdot\nabla\sqrt{\Gamma},\frac{\varepsilon\Gamma^\frac{3}{2}}{\varepsilon\Gamma+U}\bigg \rangle|\leq 2\|B\|_{{\rm BMO}}\|\nabla\sqrt{\Gamma}\|_2\|\nabla\frac{\varepsilon\Gamma^\frac{3}{2}}{\varepsilon \Gamma + U}\|_2,
$$
and
\begin{align*}
\big\|\nabla\frac{\varepsilon\Gamma^\frac{3}{2}}{\varepsilon \Gamma + U}\big\|_2 &= \big\|\frac{3\varepsilon\Gamma \nabla\sqrt{\Gamma}}{\varepsilon \Gamma + U} - \frac{\varepsilon\Gamma}{\varepsilon \Gamma + U}\sqrt{\Gamma}\nabla\log V\big\|_2\\
&\leq 3\|\nabla\sqrt{\Gamma}\|_2+\|\sqrt{\Gamma}\nabla\log V\|_2 \leq \frac{3}{2}\frac{\sqrt{d}}{\sqrt{2\beta(t-t_s)}}+\mathcal N_1^{\frac{1}{2}}(\tau).
\end{align*}
Thus, upon applying quadratic inequality, we have
$$
|\bigg \langle \Gamma \frac{\tilde{b} \cdot \varepsilon\nabla \Gamma}{\varepsilon \Gamma + U} \bigg \rangle| \leq \gamma \mathcal N_1(\tau) + \frac{1}{\gamma}\|B\|_{\rm BMO}^2\frac{d}{8\beta(t-t_s)} + \|B\|_{\rm BMO}\frac{d}{2\beta(t-t_s)}.
$$

\begin{remark}
\label{P_remark}
Instead of $\tilde{b}$ above we could have considered $\tilde{b}+\check{b}$ with $\check{b}=\nabla P$, $P \in [L^\infty(\mathbb R_+ \times \mathbb R^d)]^{d \times d}$, with ${\rm div\,}\check{b}$ form-bounded and in the Kato class. 
Indeed, we could modify Step 1 as follows:
\begin{align*}
|\langle \Gamma \check{b}\cdot\nabla\log V\rangle| & \leq \langle ({\rm div\,}\check{b})_-, -\Gamma \log V\rangle +2 \|P\|_\infty \|\nabla^2\sqrt{\Gamma}\|_2 \|\sqrt{\Gamma}\log V \|_2 \\
& + 2\|P\|_\infty\|\nabla \sqrt{\Gamma}\|_2\|\nabla(\sqrt{\Gamma}\log V)\|_2.
\end{align*}
We estimate the first term in the same way as $A_{\rm div}(\tau)$ below, in the last two terms use $\|\nabla^2\sqrt{\Gamma}\|_2 \leq C(t-t_s)^{-1}$, $\|\nabla \sqrt{\Gamma}\|_2 \leq C(t-t_s)^{-\frac{1}{2}}$ so that we can argue as in Step 1 above. Step 2 is modified similarly.
\end{remark}

3.~Term ``$\langle \varepsilon \Gamma \frac{\hat{b} \cdot\nabla \Gamma}{\varepsilon\Gamma + U} \rangle$''. We have
\begin{align*}
\left |\bigg \langle \varepsilon \Gamma \frac{\hat{b} \cdot\nabla \Gamma}{\varepsilon\Gamma + U} \bigg \rangle\right | &  \leq \langle |\hat{b}| |\nabla \Gamma| \rangle \leq \frac{\sqrt{2}}{\sqrt{\beta(t-t_s)}}\langle |\hat{b}|\Gamma_{2\beta}\rangle,
\end{align*}
where, using that $\hat{b} \in \mathbf{MF}_\delta$ with multiplicative bound $\delta \equiv \delta_{\hat{b}}$ and function $g \equiv g_{\hat{\delta}}$, we estimate:
\begin{align*}
\langle |\hat{b}|\Gamma_{2\beta}\rangle & \leq \delta\|\nabla\sqrt{\Gamma_{2\beta}}\|_2\|\sqrt{\Gamma_{2\beta}}\|_2 + g(\tau)\\
& = \frac{\delta}{2}\frac{\sqrt{d}}{\sqrt{2\beta(t-t_s)}}+g(\tau).
\end{align*} 
Thus,
$$
\left |\bigg \langle \varepsilon \Gamma \frac{\hat{b} \cdot\nabla \Gamma}{\varepsilon\Gamma + U} \bigg \rangle\right | \leq \frac{\delta}{2}\frac{\sqrt{d}}{\beta(t-t_s)}+\frac{\sqrt{2}g(\tau)}{\sqrt{\beta(t-t_s)}}.
$$

\smallskip

4. Term ``$\langle \Gamma ({\rm div\,}\hat{b})_-  \rangle$''. Since $({\rm div\,}\hat{b})_-$ is form-bounded with form-bound $\nu \equiv \nu_-$ and function $h \equiv h_\nu$, we have
\begin{align*}
\langle \Gamma ({\rm div\,}\hat{b})_-  \rangle & \leq \frac{\nu}{4}\frac{d}{2 \beta} \frac{1}{t-t_s} + h(\tau).
\end{align*}

\medskip

5.~Term ``$\big \langle \Gamma \hat{b} \cdot \nabla \log V \big \rangle$''. This is the most difficult term. Using $\hat{b} \in \mathbf{MF}_\delta$ and the form-boundedness of ${\rm div\,}\hat{b}_-$, we will prove that there exists a constant $\tilde{c}$ such that
\begin{align}
 |\big \langle \Gamma \hat{b} \cdot \nabla \log V \big \rangle|  & \leq (2\gamma + \nu \sigma^{-1} )\mathcal N \notag \\
&+ \bigg(\frac{c^*_0}{t-t_s} + \frac{g(\tau)}{2 \sqrt{\beta (t-t_s)}}  + \frac{c^*_2}{t-t_s} + \frac{d \;g(\tau)}{4 \sqrt{\beta (t-t_s)}} + \frac{\nu d}{4 \beta (t-t_s)} + h(\tau)\bigg) (-Y(\tau) - c) \notag \\
& + \frac{c^*_1 + c_3^*+ \frac{\nu d}{16 \beta}}{t-t_s}  + \frac{d}{4\gamma \sigma}g^2(\tau). \label{AAA_est}
\end{align}

Proof of \eqref{AAA_est}.
We estimate
\begin{align*}
|\big \langle \Gamma \hat{b} \cdot \nabla \log V \big \rangle| & \leq |\big \langle \hat{b} \cdot \nabla \Gamma , -\log V \big \rangle|  + \langle \Gamma ({\rm div}\hat{b})_-, -\log V \rangle \\
& \leq \frac{1}{\sqrt{\beta (t-t_s)}} \big \langle |\hat{b}| \Gamma, -\log V \big \rangle^{1/2} \bigg \langle |\hat{b}| \frac{|o-\cdot|^2}{4 \beta (t-t_s)} \Gamma, -\log V \bigg \rangle^{1/2} + \langle \Gamma ({\rm div}\hat{b})_-, -\log V \rangle\\
& =: \frac{1}{\sqrt{\beta (t-t_s)}} A_0^{1/2} A_2^{1/2} + A_{{\rm div}} \\
& \leq \frac{1}{2 \sqrt{\beta (t-t_s)}}(A_0+A_2)+A_{{\rm div}},
\end{align*}
where
$$
A_0(\tau):= \big \langle |\hat{b}| \Gamma (-\log V) \big \rangle=:\langle |\hat{b}|\varphi \rangle,
$$
$$
A_2(\tau):=\bigg \langle |\hat{b}| \frac{|o-\cdot|^2}{4 \beta (t-t_s)} \Gamma (-\log V) \bigg \rangle=:\langle |\hat{b}|\psi\rangle,
$$
$$
A_{{\rm div}}(\tau):=\langle  ({\rm div}\hat{b})_-, \Gamma (-\log V) \rangle=:\langle ({\rm div}\hat{b})_-\varphi\rangle.
$$

 Let us estimate $A_0(\tau).$
\begin{align*}
A_0(\tau) = & \big \langle |b| \Gamma (-\log V) \big \rangle \equiv \langle |b| \varphi \rangle \\
& \leq \delta \|\nabla \sqrt{\varphi} \|_2 \langle \varphi \rangle^{1/2} + g \langle \varphi \rangle \\
& =\frac{1}{2} \delta \big \langle (\nabla \varphi)^2 /\varphi \big \rangle^{1/2} \langle \varphi \rangle^{1/2} + g \langle \varphi \rangle,
\end{align*}
\[
\nabla \varphi = \bigg( \frac{\nabla \Gamma}{\Gamma} +\frac{\nabla \log V}{\log V} \bigg) \varphi, \qquad \frac{(\nabla \varphi)^2}{\varphi} =  \bigg( \frac{\nabla \Gamma}{\Gamma} +\frac{\nabla (-\log V)}{-\log V} \bigg)^2 \varphi,
\]
\begin{align*}
\frac{(\nabla \varphi)^2}{\varphi} & \leq 2 \bigg( \frac{(\nabla \Gamma)^2}{\Gamma} (-\log V) +\frac{(\nabla \log V)^2}{-\log V} \Gamma \bigg) \\
& \leq 2 \bigg( \frac{|o-\cdot|^2}{(2\beta(t-t_s))^2}\Gamma (-\log V) + \Gamma (\nabla \log V)^2 \bigg) \;\;(\text{ because } - \log V >1 ).
\end{align*}
By the equality $\frac{|o-\cdot|^2}{4 \beta (t-t_s)} \Gamma = \beta (t-t_s) \Delta \Gamma +\frac{d}{2} \Gamma,$
\begin{align*}
\frac{1}{2}\bigg \langle \frac{(\nabla \varphi)^2}{\varphi} \bigg \rangle & \leq \big \langle \Gamma (\nabla \log V )^2 \big \rangle + \frac{1}{\beta (t-t_s)}\big \langle \big(\beta (t-t_s) \Delta \Gamma + \frac{d}{2} \Gamma \big)(-\log V) \big \rangle \\
& \leq \sigma^{-1} \mathcal{N} +\big \langle \nabla \Gamma, \nabla \log V \big \rangle + \frac{d}{2 \beta(t-t_s) } \langle \varphi \rangle \\
& \leq 2 \sigma^{-1} \mathcal{N} +\frac{1}{4} \bigg\langle \frac{(\nabla \Gamma)^2}{\Gamma} \bigg \rangle + \frac{d}{2 \beta (t-t_s)} \langle \varphi \rangle\\
& \leq 2 \sigma^{-1} \mathcal{N} + \frac{d}{8 \beta (t-t_s)}+ \frac{d}{2 \beta (t-t_s)} \langle \varphi \rangle.
\end{align*}
Therefore, by inequalities $(B+C+D)^{1/2}\leq (B+C)^{1/2}+D^{1/2}$ and $E^{1/2}(B+C)^{1/2}M^{1/2} \leq (B+C)\varepsilon + (4 \varepsilon)^{-1} E M$ for positive numbers with $\varepsilon = \sigma\gamma/2,$
\begin{align*}
\frac{A_0(\tau)}{2 \sqrt{\beta (t-t_s)}} & \leq \frac{\delta}{4 \sqrt{\beta (t-t_s)}} \bigg( 2\sigma^{-1} \mathcal{N}(\tau)+ \frac{d}{8 \beta (t-t_s)}+ \frac{d}{2 \beta (t-t_s)} \langle \varphi \rangle \bigg)^{1/2}\langle \varphi \rangle^{1/2} + \frac{g(\tau)}{2 \sqrt{\beta (t-t_s)}}\langle \varphi \rangle \\
& \leq \gamma \mathcal{N}(\tau) + \bigg(\frac{c^*_0}{t-t_s} + \frac{g(\tau)}{2 \sqrt{\beta (t-t_s)}} \bigg) \langle \varphi \rangle + \frac{c^*_1}{t-t_s},
\end{align*}
where $c^*_i =c^*_i(d,\sigma, \xi, \delta, \gamma) > 0$, $i=0,1$.

Analogous calculations show that there are constants $c^*_i =c^*_i(d,\sigma, \xi, \delta, \gamma) > 0,$ $i=2,3$, such that
\[
\frac{A_2(\tau)}{2 \sqrt{\beta (t-t_s)}} \leq \gamma \mathcal{N}(\tau) + \bigg(\frac{c^*_2}{t-t_s} + \frac{d \;g(\tau)}{4 \sqrt{\beta (t-t_s)}} \bigg) \langle \varphi \rangle + \frac{c^*_3}{t-t_s} +\frac{d}{4\gamma\sigma}g^2(\tau).\label{A2_est}\tag{$\bullet$}
\]

Indeed, $A_2(\tau)\leq \frac{1}{2}\delta\big\langle\frac{(\nabla\psi)^2}{\psi}\big\rangle^\frac{1}{2}\langle \psi\rangle^\frac{1}{2} + g(\tau)\langle \psi\rangle$, $\psi:=\frac{|o-\cdot|^2}{4\beta(t-t_s)}\Gamma(-\log V)$.
\[
\nabla\psi=\frac{\cdot-o}{2\beta(t-t_s)}\Gamma(-\log V)+\frac{|o-\cdot|^2}{4\beta(t-t_s)}\nabla\Gamma(-\log V)+\frac{|o-\cdot|^2}{4\beta(t-t_s)}\Gamma\nabla(-\log V),
\]
\[
(\nabla\psi)^2\leq 3\bigg(\frac{|o-\cdot|^2}{4\beta^2(t-t_s)^2}\Gamma^2(\log V)^2+\frac{|o-\cdot|^4}{16\beta^2(t-t_s)^2}(\nabla\Gamma)^2(\log V)^2 + \frac{|o-\cdot|^4}{16\beta^2(t-t_s)^2}\Gamma^2(\nabla\log V)^2 \bigg),
\]
\[
\frac{(\nabla\psi)^2}{\psi}\leq 3\bigg( \frac{1}{\beta(t-t_s)}\Gamma (-\log V) +\frac{|o-\cdot|^2}{4\beta(t-t_s)}\frac{(\nabla\Gamma)^2}{\Gamma}(-\log V)+ \frac{|o-\cdot|^2}{4\beta(t-t_s)}\Gamma\frac{(\nabla\log V)^2}{-\log V}\bigg),
\]
Using inequality $-\log V  > \frac{|o-\cdot|^2}{4 \beta (t-t_s)}$ and equality $\frac{(\nabla\Gamma)^2}{\Gamma} =\Delta\Gamma +\frac{d}{2\beta(t-t_s)}\Gamma$ we have
\[
\bigg \langle\frac{(\nabla\psi)^2}{\psi}\bigg\rangle\leq 3\bigg( \frac{1}{\beta(t-t_s)}\langle\varphi\rangle+\frac{d}{2\beta(t-t_s)}\langle\psi\rangle+ \bigg\langle\frac{|o-\cdot|^2}{4\beta(t-t_s)}\Delta\Gamma,(-\log V)\bigg\rangle+\langle\Gamma(\nabla\log V)^2\rangle\bigg),
\]
\begin{align*}
\bigg\langle\frac{|o-\cdot|^2}{4\beta(t-t_s)}\Delta\Gamma,(-\log V)\bigg\rangle &=\bigg\langle\frac{o-\cdot}{2\beta(t-t_s)}\cdot\nabla\Gamma,-\log V\bigg\rangle - \bigg\langle\frac{|o-\cdot|^2}{4\beta(t-t_s)}\nabla\Gamma,\nabla(-\log V)\bigg\rangle\\
&\leq \bigg\langle\frac{|o-\cdot|^2}{4\beta^2(t-t_s)^2}\Gamma(-\log V)\bigg\rangle + \bigg\langle \frac{|o-\cdot|^3}{8\beta^2(t-t_s)^2}\Gamma^\frac{1}{2},
\Gamma^\frac{1}{2}|\nabla\log V|\bigg\rangle\\
&\leq \frac{1}{\beta(t-t_s)}\langle\psi\rangle + \frac{C(d)}{\sqrt{\beta(t-t_s)}}(\sigma^{-1}\mathcal{N})^\frac{1}{2}
\end{align*}
Thus
\[
\bigg \langle\frac{(\nabla\psi)^2}{\psi}\bigg\rangle\leq 3\bigg( \frac{1}{\beta(t-t_s)}\langle\varphi\rangle+\frac{d+2}{2\beta(t-t_s)}\langle\psi\rangle + 2\sigma^{-1}\mathcal{N} + \frac{C^2(d)}{4\beta(t-t_s)}\bigg),
\]
Also $\langle\psi\rangle=\beta(t-t_s)\big\langle\frac{(\nabla\Gamma)^2}{\Gamma}(-\log V)\big\rangle=\frac{d}{2}\langle\varphi\rangle+\beta(t-t_s)\langle\Delta\Gamma,-\log V\rangle$ and so $$\langle\psi\rangle\leq \frac{d}{2}\langle\varphi\rangle+\bigg(\frac{d}{2}\beta(t-t_s)\bigg)^\frac{1}{2}(\sigma^{-1}\mathcal{N})^\frac{1}{2},$$
\[
\bigg \langle\frac{(\nabla\psi)^2}{\psi}\bigg\rangle\leq 3\bigg(3\sigma^{-1}\mathcal{N} + \frac{(d+2)^2}{4\beta(t-t_s)}\langle\varphi\rangle +  \frac{C'(d)}{4\beta(t-t_s)} \bigg),
\]
and $(\bullet)$ follows. 

\begin{remark}
Estimate \eqref{A2_est} requires $g \in L^2(\mathbb R_+)$. Everywhere else in the proof it suffices to assume \eqref{g_cond2}. 
\end{remark}

Finally, since $({\rm div\,}\hat{b})_-$ is form-bounded,
\begin{align*}
A_{{\rm div}}(\tau) & \leq \frac{\nu}{4}\left\langle \frac{(\nabla \varphi)^2}{\varphi} \right\rangle  + h(\tau)\langle \varphi \rangle \\
& \leq \nu\sigma^{-1} \mathcal{N} + \bigg(\frac{\nu d}{4 \beta (t-t_s)} +  h(\tau)\bigg) \langle \varphi \rangle + \frac{\nu d}{16 \beta (t-t_s)}
\end{align*}
Therefore,
\begin{align*}
& \frac{1}{2 \sqrt{\beta (t-t_s)}}(A_0+A_2)+A_{{\rm div}}  \\
& \leq (2\gamma + \nu\sigma^{-1} ) \mathcal N  + \bigg(\frac{c^*_0}{t-t_s} + \frac{g(\tau)}{2 \sqrt{\beta (t-t_s)}}  + \frac{c^*_2}{t-t_s} + \frac{d \;g(\tau)}{4 \sqrt{\beta (t-t_s)}} + \frac{\nu d}{4 \beta (t-t_s)} + h(\tau)\bigg) \langle \varphi \rangle \\
& + \frac{c^*_1 + c_3^*+ \frac{\nu d}{16 \beta}}{t-t_s}  + \frac{d}{4\gamma \sigma}g^2(\tau).
\end{align*}
The latter gives \eqref{AAA_est} upon noticing that $\langle \varphi \rangle = -Y(\tau) + \frac{d}{2}\log \frac{t-\tau}{t-t_s} - \log c_0 \leq -Y(\tau)-\tilde{c}$ for $\tilde{c}=\frac{d}{2}\log 2 +\log c_0$. 
This ends the proof of 5.

\medskip

We are in position to complete the proof of Lemma \ref{cl2}.
By estimates 1-5,
\begin{align*}
J & \geq -\big[2\gamma + (2\gamma+ \nu)\sigma^{-1}\big] \mathcal N(\tau) - \frac{C_1}{t-t_s} - \frac{(d+2)g(\tau)}{4\sqrt{\beta(t-t_s)}} - \frac{d}{4\gamma \sigma} g^2(\tau) - h(\tau) \\
& - \biggl(\frac{C_2}{t-t_s}+\frac{\sqrt{2} g(\tau)}{\sqrt{\beta(t-t_s)}} + h(\tau) \biggr)(-Y(\tau) - \tilde{c}).
\end{align*}
 Therefore, 
\begin{align*}
\mathcal{M}(\tau) & \geq (1-(4+2 \sigma^{-1}) \gamma - \sigma^{-1}\nu) \mathcal{N}(\tau) - ( Y(\tau) - c) \partial_\tau \mu(\tau) \\
& - \frac{C_1}{t-t_s} - \frac{(d+2)g(\tau)}{4\sqrt{\beta(t-t_s)}} - \frac{d}{4\gamma \sigma} g^2(\tau) - h(\tau) - \biggl(\frac{C_2}{t-t_s}+\frac{\sqrt{2} g(\tau)}{\sqrt{\beta(t-t_s)}} + h(\tau) \biggr)(-Y(\tau) - \tilde{c}), 
\end{align*}
where, recalling that $\nu<\sigma$, we select $\gamma>0$ sufficiently small to keep the coefficient of $\mathcal N(\tau)$ non-negative.
Next, recall that $$\partial_\tau \mu(\tau)=\frac{1}{t-t_s}+ \frac{\sqrt{2} g(\tau)}{\sqrt{\beta(t-t_s)}}+ g^2(\tau) +  h(\tau).$$
It is now easily seen that we can select $c$ sufficiently large so that $\mathcal M(\tau) \geq 0$.
The proof of Lemma \ref{cl2} is completed.

\subsection{$\hat{G}$-function for $-\nabla \cdot a \cdot \nabla + \nabla \cdot (\hat{b}+\tilde{b})$}

Let $u_*(t,x;s,y)$ denote the heat kernel of $$\Lambda_{*}=-\nabla \cdot a \cdot \nabla + \nabla \cdot (\hat{b}+\tilde{b}).$$
By \eqref{ub_11}, by duality, $u_*(t,x;s,y)$ satisfies the upper Gaussian  bound
\begin{equation}
\label{ub_12}
u_*(t,x;s,y) \leq \hat{c}_3 k_{c_4}(t-s;x-y).
\end{equation}
The constants in the next proposition depend on the same parameters as the constants in the theorem except of the Kato bounds $\mu_{\pm}$ and $({\rm div\,}b)_{\pm}$.

\begin{proposition}
\label{g2_thm_3}
Let $\beta$ and $\mathbb{C}$ be constants from Proposition \ref{g1_thm_3}. 
Set $o=\frac{x+y}{2}$, $x,y \in \mathbb R^d$, $t_s=\frac{t+s}{2}$.
Then
\[
\hat{G}(t_s) 
:=\langle k_\beta(t_s-s,o -\cdot)\log u_*(t_s, \cdot; s,z) \geq - \tilde{Q}(t_s-s) - \mathbb{C}, \quad z \in B(o,\sqrt{t_s-s}).
\]
\end{proposition}

\begin{proof}
The proof repeats the proof of Proposition \ref{g1_thm_3}, except that we have to deal with the positive part ${\rm div\,}\hat{b}_+$ of the divergence of $\hat{b}$.
\end{proof}

\subsection{Auxiliary operator $-\nabla \cdot a \cdot \nabla +  (\hat{b} +\tilde{b}) \cdot \nabla -  ({\rm div\,}\hat{b})_-$}

Recall the notation
$\Lambda=-\nabla \cdot a \cdot \nabla + (\hat{b} + \tilde{b}) \cdot \nabla$, $\Lambda_{*}=-\nabla \cdot a \cdot \nabla + \nabla \cdot (\hat{b}+\tilde{b})$.
Set $$\mathcal  H^- := \Lambda - (\mydiv \hat{b})_-.$$ Let $H^{t,s}f$ denote the solution of
\[ 
	\left\{ \begin{array}{rcl}
		-\frac{d}{d t} H^{t,s} f = \mathcal  H^- H^{t,s} f & ,  \\
	0 \leq f \in L^1 \cap L^\infty. & 
	\end{array} \right.
\]
Let $h(t):= H^{t,s}f.$ It is seen (for example, using the Duhamel formula) that $$u(t,x;t_s,y) \leq h(t,x;t_s,y), 
$$
$$ u_*(t_s,x;s,y) \leq h(t_s,x;s,y),$$ where $u, u_*$ are the heat kernels of $\Lambda$, $\Lambda_{*}$, respectively. It is seen that
\[
h(t,x;s,y) \geq (4 \pi \beta (t-t_s))^{d/2} \langle k_\beta (t-t_s,o -\cdot) h(t,x;t_s,\cdot) h(t_s,\cdot;s,y) \rangle,
\]
\[
k_\beta (t-t_s,o -\cdot)=k_\beta (t_s-s,o -\cdot),
\]
and, for all $2 |x-y| \leq \sqrt{\beta(t-t_s)}\;,$ due to Proposition \ref{g1_thm_3} and Proposition \ref{g2_thm_3},
\begin{align*}
\log h(t,x;s,y) & \geq \log (4 \pi \beta)^{d/2} + \tilde{Q}(t-t_s)\\
& + \langle k_\beta (t-t_s,o-\cdot) \log u(t,x;t_s,\cdot) \rangle + \langle k_\beta (t-t_s,o-\cdot) \log u_*(t_s,\cdot; s,y) \rangle \\
& \geq \log (4 \pi \beta)^{d/2} - \tilde{Q}(t-t_s)  -2 \mathbb{C} \\
& = - \tilde{Q}(t-s)  -2 \mathbb{C} + \log (8 \pi \beta)^{d/2},
\end{align*}
i.e.\,we have proved a lower Gaussian bound for $h(t,x;s,y)$ but only for  $2 |x-y| \leq \sqrt{\beta(t-t_s)}$.
Now, the standard argument  (``small gains yield large gain''), see e.g.\,\cite[Theorem 3.3.4]{Da}, gives

\begin{theorem}
\label{lb_div}
There exist constants $c_1$, $c_2 > 0$ such that, for all $x,y \in \mathbb{R}^d$, $0 \leq s<t<\infty$,
\begin{equation}
c_1 k_{c_2}(t-s,x-y) \leq h(t,x;s,y)
\tag{$\mbox{LGB}^{h_-}$.}
\end{equation}

\end{theorem}

\subsection{Proof of Theorem \ref{thm2} (a priori lower bound)} 

\label{thm2_proof}

\textit{Step 1.~}First, we establish an upper bound on the heat kernel  $h(t,x;s,y)$ of $\mathcal  H^-=A+(\hat{b}+\tilde{b})\cdot \nabla - ({\rm div\,} \hat{b})_-$:
\begin{equation}
\label{temp_ub2}
\tag{$\ast\ast$}
h(t,x;s,y) \leq \tilde{c}_3 k_{c_4}(t-s;x-y)
\end{equation}
for all $x, y \in \mathbb{R}^d$ and $0 \leq s < t < \infty$.
Indeed, we can write the Duhamel series for $h(t,x;s,y)$, with $\mathcal  H^-$ is viewed as $A+(\hat{b}+ \tilde{b})\cdot \nabla$ perturbed by potential $-({\rm div\,}\hat{b})_-$. Then the upper Gaussian  bound on $u(t,x;s,y)$, established in Theorem \ref{thm0_}, and the hypothesis that the Kato norm $\mu_-$ of $({\rm div} \hat{b})_-$ is sufficiently small, yield \eqref{temp_ub2} via a standard argument.

\textit{Step 2.~}Let us consider $A+(\hat{b} + \tilde{b}) \cdot \nabla$ as the perturbation of $\mathcal  H^-$ by $ ({\rm div\,}\hat{b})_-$. Then
$$
u(t,x;s,y)=h(t,x;s,y) - \int_s^t \big\langle u(t,x;\tau,\cdot) ({\rm div\,}\hat{b})_-(\cdot)h(\tau,\cdot;s,y)\big\rangle d\tau.
$$
We apply  Theorem \ref{lb_div}  to the first term. In the second term, we can apply the upper Gaussian  bound on $u(t,x;\tau,\cdot)$ established in Theorem \ref{thm0_}, apply \eqref{temp_ub2}, and use the hypothesis that the Kato norm $\mu_-$ of $ {\rm div\,}\hat{b}_-$ is sufficiently small  to obtain
$$u(t,x;s,y) \geq \frac{C_1}{(t-s)^{\frac{d}{2}}}-\frac{C_2\mu_-}{(t-s)^{\frac{d}{2}}}$$
whenever  $|x-y| \leq \sqrt{t-s}$, for all $0 \leq s<t <\infty$ (see \cite[Sect.\,5]{Z} for detailed argument).
This yields $u(t,x;s,y) \geq C_0(t-s)^{-\frac{d}{2}}$ provided that $\mu_-$ is sufficiently small. 
Now, a standard argument (\cite[Theorem 3.3.4]{Da})  gives the required lower Gaussian bound on $u(t,x;s,y)$ for all $x,y \in \mathbb R^d$, $0 \leq s<t <\infty$.

\bigskip

\section{Proof of Theorem \ref{thm0__}}

\label{approx_sect}

We fix the following bounded smooth approximation of $a$, $\hat{b}$, $\tilde{b}$. In what follows, we extend $a$, $\hat{b}$, $\tilde{b}$ to $t<0$ by zero. Let $E_{\varepsilon}$ be the De Giorgi mollifier on $\mathbb R \times \mathbb R^d$,
$$
E_{\varepsilon} f(t,x):=\int_{\mathbb R }\langle e^{\varepsilon \Delta_{d+1}}(t,x;s,\cdot)f(s,\cdot)\rangle ds , \quad f \in L^1_{\loc}(\mathbb R \times \mathbb R^d)$$
($\Delta_{d+1}$ is the Laplacian on $\mathbb R \times \mathbb R^d$). Let
$$E_\varepsilon^d f(x):=\langle e^{\varepsilon \Delta}(x,\cdot)f(\cdot)\rangle, \quad f \in L^1_{\loc}.$$
For a Banach-valued measurable function $h=h(t)$, we define its Steklov averaging
$$
[h]_\varepsilon(t):=\frac{1}{\varepsilon}\int_t^{t+\varepsilon}h(r)dr.
$$

Put
$$
a_{\varepsilon_1}:=E_{\varepsilon_1}a.
$$
Clearly, $a_{\varepsilon_1}$ is $C^\infty$ smooth, symmetric and uniformly elliptic with the same constants as $a$.

Set
$$
\hat{b}_\varepsilon:=\big[E_\varepsilon^d \hat{b}\big]_{c(\varepsilon)}, 
$$
$$
{(\rm div\,}\hat{b})_{\pm,\varepsilon}:=[E_\varepsilon^d {(\rm div\,}\hat{b})_\pm]_{c(\varepsilon)},
$$
where $c(\varepsilon)>0$ is to be chosen.
Then, clearly, ${(\rm div\,}\hat{b})_\varepsilon={(\rm div\,}\hat{b})_{+,\varepsilon} - {(\rm div\,}\hat{b})_{-,\varepsilon}$.

(Note that we can not use the same regularization of $\hat{b}$ as in Proposition \ref{fbd_approx_prop} since  the indicator function $\mathbf{1}_\varepsilon$ there would not allow us to control ${\rm div\,}\hat{b}_\varepsilon$.)

Finally, given a vector field $(\tilde{b}  \in L^\infty\,{\rm BMO}^{-1},{\rm div\,}\tilde{b}=0)$, we define its bounded smooth approximation as in {\cite[Sect.\,3]{QX}}. 
There exist a skew-symmetric matrix $B \in [L^\infty\,{\rm BMO}]^{d\times d} \cap [L_{\loc}^p(\mathbb R_+ \times \mathbb R^d)]^d$ for all $1 \leq p<\infty$ such that $b=\nabla B$.
Set 
$$
B_\varepsilon:=e^{\varepsilon \Delta_{d+1}}(B \wedge U_\varepsilon \vee V_\varepsilon) \quad (\text{max and min are taken component-wise}), 
$$ 
where $U_\varepsilon:=(-c\log|x| + \varepsilon^{-1}) \wedge \varepsilon^{-1} \vee 0$, $V_\varepsilon:=(c\log|x|-\varepsilon^{-1}) \wedge 0 \vee (-\varepsilon^{-1})$ are ${\rm BMO}$ functions with compact support. The constant $c$ is chosen so that $\|c\log|x|\|_{\rm BMO} \leq \|B\|_{L^\infty\,{\rm BMO}}$.
Define
$$
\tilde{b}_\varepsilon:=\nabla B_\varepsilon
$$
(since $B_\varepsilon$ are skew-symmetric, ${\rm div\,}\tilde{b}_\varepsilon=0$).

\begin{proposition}
\label{lem_hat_b}
Let $\hat{b} \in \mathbf{MF}_\delta$. Then the following are true:

\smallskip

{\rm (\textit{i})} For every $t \geq 0$, $x \in \mathbb R^d$,
$$
|E_\varepsilon^d\hat{b}(t,x)| \leq \sqrt{\frac{d}{8\varepsilon}}\delta  + g(t).
$$

\smallskip

{\rm (\textit{ii})} $E_\varepsilon^d\hat{b} \in L_\loc^2(\mathbb R_+,C_b(\mathbb R^d))$.
\end{proposition}
\begin{proof}(\textit{i}) We have
\begin{align*}
|E_\varepsilon^d\hat{b}(t,x)| & =|\langle \hat{b}(t,\cdot)\sqrt{e^{\varepsilon\Delta}(x,\cdot)},\sqrt{e^{\varepsilon\Delta}(x,\cdot)}\rangle | \\
& \leq \delta \langle \big|\nabla \sqrt{e^{\varepsilon\Delta}(x,\cdot)}\big|^2 \rangle^{\frac{1}{2}}   + g(t)  \\
& (\text{we use $\big\langle \big|\nabla\sqrt{e^{\varepsilon\Delta}(x,\cdot)}\big|^2\big\rangle=\frac{d}{8\varepsilon}$}) \\
& \leq \sqrt{\frac{d}{8\varepsilon}}\delta  + g(t).
\end{align*}

(\textit{ii}) Since $E_\varepsilon^d=E^d_{\varepsilon/2}E^d_{\varepsilon/2}$, we have for a.e.\,$t \in \mathbb R_+$ $E_\varepsilon^d\hat{b}(t,\cdot) \in C_b$. Since $g \in L^2(\mathbb R_+)$, we have $E_\varepsilon^d\hat{b} \in L_\loc^2(\mathbb R_+,C_b(\mathbb R^d))$.
\end{proof}

Recall that
$$
\hat{b}_\varepsilon=\big[E_\varepsilon^d \hat{b}\big]_{c(\varepsilon)}, 
$$
where the rate of Steklov averaging $c(\varepsilon) \downarrow 0$ is to be chosen.

\begin{proposition} 
\label{mf_reg}
Let $\hat{b} \in \mathbf{MF}_\delta$. Then

\smallskip

{\rm (\textit{i})} $\hat{b}_\varepsilon\in [L^\infty(\mathbb R_+ \times \mathbb R^d)\cap C^\infty(\mathbb R_+ \times \mathbb R^d)]^d$.

{\rm (\textit{ii})} $\hat{b}_\varepsilon \in \mathbf{MF}_\delta$ with $\|g_\varepsilon\|_{L^{2}(\mathbb R_+)} \leq \|g\|_{L^{2}(\mathbb R_+)}$.
\end{proposition}

\begin{proof} 
{(\textit{i})} follows immediately from Proposition \ref{lem_hat_b}(\textit{ii}) and the properties of Steklov averaging.

\smallskip

(\textit{ii})  
First, let us note that $\|\nabla\sqrt{E^d_\varepsilon|\varphi|^2}\|_2 \leq \|\nabla \varphi\|_2$. Indeed,
\begin{align*}
\|\nabla\sqrt{E^d_\varepsilon|\varphi|^2}\|_2 & =\big\|\frac{E^d_\varepsilon(|\varphi||\nabla|\varphi|)}{\sqrt{E^d_\varepsilon|\varphi|^2}}\big\|_2 \leq \|\sqrt{E^d_\varepsilon|\nabla |\varphi||^2}\|_2=\|E_\varepsilon|\nabla |\varphi||^2\|_1^\frac{1}{2} \leq\|\nabla|\varphi|\|_2\leq \|\nabla \varphi\|_2.
\end{align*}
By $\hat{b} \in \mathbf{MF}_\delta$, we have for a.e.\,$t \in \mathbb R_+$ and all $\varphi \in \mathcal S$,
\begin{align*}
 \langle |E_\varepsilon^d\hat{b}(t)| \varphi,\varphi\rangle & \leq \langle |b(t)|,E^d_\varepsilon |\varphi|^2\rangle \\
& \leq \delta \|\nabla \sqrt{E^d_\varepsilon |\varphi|^2}\|_2\|\sqrt{E^d_\varepsilon|\varphi|^2}\|_2 + g(t)\|\sqrt{E^d_\varepsilon|\varphi|^2}\|_2 \\
& \leq \delta \|\nabla \varphi\|_2\|\varphi\|_2 + g(t)\|\varphi\|_2^2.
\end{align*}
Hence for a.e.\,$t \in \mathbb R_+$
\begin{align*}
 \langle |\big[ E_\varepsilon^d\hat{b}\big]_{c(\varepsilon)}(t)| \varphi,\varphi\rangle & \leq   \big[\langle|E_\varepsilon^d\hat{b}| \varphi,\varphi\rangle\big]_{c(\varepsilon)} (t) \\
& \leq \delta \|\nabla \varphi\|_2\|\varphi\|_2 + [g]_{c(\varepsilon)}(t)\|\varphi\|_2^2.
\end{align*}
Hence $\hat{b}_\varepsilon \in \mathbf{M}_{\delta,g_{\varepsilon}}$ with $g_\varepsilon:= [g]_{c(\varepsilon)}$, and, clearly, $\|g_\varepsilon\|_2 \leq \|g\|_2$.
\end{proof}

Recall
$$
({\rm div\,}\hat{b})_{\pm,\varepsilon}:=[E_\varepsilon^d ({\rm div\,}\hat{b})_\pm]_{c(\varepsilon)}.
$$

\begin{proposition} \label{form_reg} Let $({\rm div\,}\hat{b})_\pm$ be form-bounded, i.e.
$({\rm div\,}\hat{b})_\pm \in L^1_{\loc}(\mathbb R_+ \times \mathbb R^d)$, the inequality
\begin{equation*}
\langle ({\rm div\,}\hat{b})_\pm \varphi,\varphi\rangle   \leq \nu_\pm \|\nabla \varphi\|^2_2  + h_\pm(t)\|\varphi\|_2^2, \quad \varphi \in W^{1,2}
\end{equation*}
holds for a.e.\,$t \in \mathbb R_+$ and some functions $0 \leq h_\pm  \in L^1_{\loc}(\mathbb R_+)$. 

Let $\mu_\pm$ be the Kato norms of $({\rm div\,}\hat{b})_\pm$.
Then the following is true:

\smallskip

{\rm (\textit{i})} $({\rm div\,}\hat{b})_{\pm,\varepsilon} \in L^\infty(\mathbb R_+ \times \mathbb R^d)\cap C^\infty(\mathbb R_+ \times \mathbb R^d)$.

\smallskip

{\rm (\textit{ii})} $({\rm div\,}\hat{b})_{\pm,\varepsilon}$ is form-bounded with the same form-bounds $\nu_\pm$ and $\|h_{\pm,\varepsilon}\|_{L^1(\mathbb R_+)} \leq \|h_\pm\|_{L^1(\mathbb R_+)}$.

\smallskip

{\rm (\textit{iii})} $({\rm div\,}\hat{b})_{\pm,\varepsilon}$ have Kato norms $\mu_\pm$, for all $\varepsilon>0$.
\end{proposition}

\begin{proof}
The proof of (\textit{i}), (\textit{ii}) follows closely the proof of Proposition \ref{mf_reg}.

(\textit{iii}) Since the translations of $({\rm div\,}b)_+(s,\cdot)$ in $s$ belong to the Kato class with the same Kato norm, we have
\begin{align*}
& \sup_{t \geq 0, x \in \mathbb R^d}\int_{0}^{t}  \langle k(t-s,x,\cdot)({\rm div\,}b)_{+,\varepsilon}(s,\cdot)\rangle ds \\
& =  \sup_{t \geq 0, x \in \mathbb R^d}\frac{1}{c(\varepsilon)}\int_0^{c(\varepsilon)} \int_0^t \langle k(t-s,x,\cdot) E_\varepsilon^d({\rm div\,}b)_+(s+r,\cdot)\rangle ds dr   \leq \frac{1}{c(\varepsilon)}\int_0^{c(\varepsilon)}  \mu_+ dr = \mu_+.
\end{align*}  
The second integral in the definition of the Kato norm is treated in the same way.
\end{proof}

Let us prove the convergence results in \eqref{all_conv}. 
The convergence $a_{\varepsilon_1} \rightarrow a$ in $[L^2_{\loc}(\mathbb R_+ \times \mathbb R^d)]^{d \times d}$ is evident. Next, 
\begin{equation}
\label{b_hat_conv}
\hat{b}_\varepsilon  \equiv [E_\varepsilon^d \hat{b}]_{c(\varepsilon)} \rightarrow \hat{b} \quad \text{ in } \quad [L^1_{\loc}(\mathbb R_+,L^1_{\loc})]^d \quad \text{ as } \varepsilon \downarrow 0
\end{equation}
provided that $c(\varepsilon) \downarrow 0$ sufficiently rapidly.
Indeed, for any $T>0$, $r>0$,
\begin{align*}
\mathbf{1}_{[0,T]}\mathbf{1}_{B(0,r)}[E_\varepsilon^d \hat{b}]_{c(\varepsilon)} = \mathbf{1}_{[0,T]}[\mathbf{1}_{B(0,r)}  E_\varepsilon^d \hat{b}]_{c(\varepsilon)}.
\end{align*}
Since $C(\bar{B}(0,r))$ is a separable Banach space, by known properties of the Steklov averaging of Banach-valued functions, we have, for every fixed $\varepsilon>0$, 
$$
\mathbf{1}_{[0,T]}[\mathbf{1}_{B(0,r)}E_\varepsilon^d  \hat{b}]_{\varepsilon_2} \rightarrow \mathbf{1}_{[0,T]}\mathbf{1}_{B(0,r)}E_\varepsilon^d  \hat{b} \quad \text{ in } \quad [L^1([0,T],C(\bar{B}(0,r)))]^d \quad \text{ as } \varepsilon_2 \downarrow 0.
$$
Therefore, for every $\varepsilon>0$ can find $\varepsilon_2$ such that 
$$
\|\mathbf{1}_{[0,T]}[\mathbf{1}_{B(0,r)}E_\varepsilon^d  \hat{b}]_{\varepsilon_2} - \mathbf{1}_{[0,T]}\mathbf{1}_{B(0,r)}E_\varepsilon^d  \hat{b}\|_{L^1([0,T],L^1(\bar{B}(0,r)))}<\varepsilon.
$$
We put $c(\varepsilon):=\varepsilon_2$.
In turn,
$$
\mathbf{1}_{[0,T]}\mathbf{1}_{B(0,r)}E_\varepsilon^d  \hat{b} \rightarrow \mathbf{1}_{[0,T]}\mathbf{1}_{B(0,r)}\hat{b} \quad \text{ in } [L^1([0,T],L^1(\bar{B}(0,r)))]^d \quad \text{ as } \varepsilon \downarrow 0.
$$
Hence
$$
\hat{b}_\varepsilon  \equiv [E_\varepsilon^d \hat{b}]_{c(\varepsilon)} \rightarrow \hat{b} \quad \text{ in } \quad [L^1([0,T],L^1(\bar{B}(0,r)))]^d \quad \text{ as } \varepsilon \downarrow 0.
$$
Our choice of $c(\varepsilon)$ depends on $T$ and $r$. It is clear however that, using a diagonal argument, we can select $c(\varepsilon)$  even smaller to have \eqref{b_hat_conv}.

The same argument yields
$$
({\rm div\,}\hat{b})_{\pm,\varepsilon} \equiv [E_\varepsilon^d {\rm div\,}\hat{b}_\pm]_{c(\varepsilon)} \rightarrow {\rm div\,}\hat{b}_\pm \quad \text{ in } [L^1_{\loc}(\mathbb R_+ \times \mathbb R^d)]^d \quad \text{ as } \varepsilon \downarrow 0
$$
(we take $c(\varepsilon) \downarrow 0$ as $\varepsilon \downarrow 0$ even more rapidly, if needed).

\begin{proposition}[{\cite[Sect.\,3]{QX}}]
\label{bmo_reg}
Let $\tilde{b} \in L^\infty\,{\rm BMO}^{-1}$, ${\rm div\,}\tilde{b}=0$. Then
\smallskip

{\rm (\textit{i})} $\tilde{b}_\varepsilon \in [L^\infty(\mathbb R_+ \times \mathbb R^d)\cap C^\infty(\mathbb R_+ \times \mathbb R^d)]^d$.

\smallskip

{\rm (\textit{ii})} $\|B_\varepsilon\|_{L^\infty\,{\rm BMO}} \leq C\|B\|_{L^\infty\,{\rm BMO}}$ for a constant $C$ that only depends on the dimension $d$, and
$$
B_\varepsilon \rightarrow B \quad \text{ in } L^p_{\loc}(\mathbb R_+ \times \mathbb R^d) \text{ for all } 1 \leq p<\infty.
$$
\end{proposition}

By the definition of $\|\cdot\|_{L^\infty\,{\rm BMO^{-1}}}$, it follows from assertion (\textit{ii}) that $$\|\tilde{b}_\varepsilon\|_{L^\infty\,{\rm BMO^{-1}}} \leq C\|\tilde{b}\|_{L^\infty\,{\rm BMO^{-1}}}.$$

We are in position to end the proof of Theorem \ref{thm0__}.
Arguing as in \cite{FS}, we obtain from the a priori two-sided Gaussian heat kernel bounds established in Theorems \ref{thm0_}, \ref{thm2}: for all $\varepsilon_1$, $\varepsilon>0$, given a  solution $v_{\varepsilon_1,\varepsilon} \in C^\infty([r-R^2,r] \times \bar{B}(z,R))$ to $(\partial_t-\nabla \cdot a_{\varepsilon_1} \cdot \nabla + b_\varepsilon \cdot \nabla)v_{\varepsilon_1,\varepsilon}=0$ in 
$]r-R^2,r[ \times B(z,R)$, where $R \leq 1$, $z \in \mathbb R^d$, one have for every $0<\alpha<1$ 
$$
|v_{\varepsilon_1,\varepsilon}(t,x)-v_{\varepsilon_1,\varepsilon}(t',x')| \leq C
\|v_{\varepsilon_1,\varepsilon}\|_{L^\infty([r-R^2,r] \times \bar{B}(z,R))}
\biggl(\frac{|t-t'|^{\frac{1}{2}} + |x-x'|}{R} \biggr)^\beta
$$
for all $(t,x)$, $(t',x') \in [r-(1-\alpha^2)R^2,r] \times \bar{B}(z,(1-\alpha) R)$ 
for some constants $C$ and $\beta \in ]0,1[$ (i.e.\,independent of $\varepsilon_1$, $\varepsilon$).
This result applies, in particular, to the heat kernel $u_{\varepsilon_1,\varepsilon}(t,x;s,y)$ of $-\nabla \cdot a_{\varepsilon_1} \cdot \nabla + b_\varepsilon \cdot \nabla$ with $s$ and $y$ fixed. Therefore, applying Arzel\`{a}-Ascoli Theorem on sets $\{(t,x) \mid t \geq s+\frac{1}{n}, |x| \leq n\}$ coupled with a diagonal argument, we can extract sequences $\varepsilon_1$, $\varepsilon \downarrow 0$ 
such that
$$
u_{\varepsilon_1,\varepsilon}(\cdot,\cdot;s,y) \rightarrow u(\cdot,\cdot;s,y) \quad \text{ uniformly on every $U_n$, $n=1,2,\dots$}
$$
to some function $u(\cdot,\cdot;s,y)$. By construction, $u(t,x;s,y)$ satisfies two-sided Gaussian bounds and is H\"{o}lder continuous in $(t,x)$. Moreover, using again two-sided Gaussian bounds and a weak compactness argument, we may assume that, for every $m=1,2,\dots$,
$$
u_{\varepsilon_1,\varepsilon} \rightarrow u \text{ weakly in } L^2(\{\frac{1}{m} \leq t-s \leq m\} \times \{|x-y| \geq \frac{1}{m}\}), 
$$
so $u(t,x;s,y)$ as a function of variables $(t,x;s,y)$ is measurable. 

Furthermore, two-sided Gaussian bounds on $u$ and a standard mollifier argument yield 
$$
\langle u(t,x;s,\cdot)f(\cdot)\rangle \rightarrow f(x) \quad \text{ as } t \downarrow s
$$
in $L^p$ ($1 \leq p <\infty$) or $C_u$ depending on where $f$ is.
We define the sought evolution family $T^{t,s} \equiv T^{t,s}(a,b)$ by
$$
T^{t,s}f(x):=\langle u(t,x;;s,\cdot)f(\cdot)\rangle.
$$
The assertions (a), (c) of the theorem now follow. Assertion (b) follows via a standard compactness argument.
The integral kernel $u(t,x;s,y)$ is, by definition, a heat kernel of the formal operator $-\nabla a \cdot \nabla + (\hat{b}+\tilde{b})\cdot \nabla$.

Finally, to prove the second statement in (d), let us regularize $\hat{b}=\hat{b}^{(1)} + \hat{b}^{(2)}$  as above. That is, put $\hat{b}^{(1)}_\varepsilon:=[E_\varepsilon\hat{b}^{(1)}]_{c(\varepsilon)}$ and note that it satisfies $$\||\hat{b}^{(1)}_\varepsilon(t)|(\lambda-\Delta)^{-\frac{1}{2}}\|_{2 \rightarrow 2} \leq \sqrt{\delta_1}$$ (by an argument similar to the one in the proof of Proposition \ref{mf_reg}), while $\hat{b}^{(2)}_\varepsilon:=[E_\varepsilon \hat{b}^{(2)}]_{c(\varepsilon)}$ satisfies $$\|(\lambda-\Delta)^{-\frac{1}{2}}|\hat{b}^{(2)}_\varepsilon(t)|\|_{\infty} \leq \sqrt{\delta_2}$$ (by the integration by parts). Hence $\hat{b}_\varepsilon:=\hat{b}^{(1)}_\varepsilon + \hat{b}^{(2)}_\varepsilon \in L^\infty\,\mathbf{F}_\delta^{\scriptscriptstyle 1/2}$, so we can apply the convergence result in Theorem \ref{thm0}, which yields the required.

\bigskip

\section{Further discussion and examples}

\label{comm_sect}

\begin{remark}The theory of operator $-\Delta + b \cdot \nabla$ is quite different from the theory of $-\nabla \cdot a \cdot \nabla + b \cdot \nabla$ with general uniformly elliptic measurable matrix $a$.
This is clear already from the existence of the Kato class of vector fields $\mathbf{K}^{d+1}_\delta$, which is specific to $-\Delta + b \cdot \nabla$. Recall that, in the time-homogeneous case (for brevity),
\begin{align*}
\text{Kato class $\mathbf{K}^{d+1}_\delta$\,: }\quad |b| \in L^1_{\loc} \quad \text{ and } \quad \|(\lambda-\Delta)^{-\frac{1}{2}}|b|\|_{\infty} \leq \sqrt{\delta}
\end{align*}
for some $\lambda=\lambda_\delta \geq 0$.
Also, recall:
$$
\text{class of form-bounded vector fields $\mathbf{F}_\delta$\,: }  \quad |b| \in L^2_{\loc} \quad \text{ and } \quad \||b|(\lambda-\Delta)^{-\frac{1}{2}}\|_{2 \rightarrow 2} \leq \sqrt{\delta}.
$$
The Kato class $\mathbf{K}^{d+1}_\delta$ provides two-sided Gaussian bounds on the heat kernel of \eqref{eq0}  \cite{Z2} and ensures uniqueness in law for the corresponding SDE, at least as long as the relative bound $\delta$ can be chosen arbitrarily small \cite{BC}. 
 The Kato class contains vector fields $b$ with $|b| \not \in L^{p}_{\loc}(\mathbb R_+ \times \mathbb R^d)$ for any $p>1$. Clearly, they cannot be form-bounded. On the other hand, the Kato class does not contain even $b=b(x)$ with $|b| \in L^d(\mathbb{R}^d)$, so the two classes are incomparable.

The Kato class condition implies, by duality, that $$\|b \cdot \nabla (\lambda-\Delta)^{-1}\|_{1 \rightarrow 1} \leq c(d)\sqrt{\delta}\equiv \sqrt{\delta_1},$$ i.e.\,that $b \cdot \nabla$ is strongly subordinate to $\lambda -\Delta$ in $L^1$. Then, if $\delta_1<1$, the Miyadera Theorem ensures that the algebraic sum $-\Delta + b \cdot \nabla$ of domain $(1-\Delta)^{-1}L^1$ generates a quasi bounded $C_0$ semigroup in $L^1$. (This is one instance where the Miyadera Theorem is indispensable.) On the other hand, if $b \in \mathbf{F}_\delta$ with $\delta<1$, then the KLMN Theorem ensures that the quadratic form of $-\Delta + b \cdot \nabla$ of domain $W^{1,2}$ determines the (minus) generator of a quasi contraction $C_0$ semigroup in $L^2$. The former semigroup cannot be a quasi contraction in $L^2$, while the latter semigroup cannot be strongly continuous in $L^1$. The bases of these solution theories of \eqref{eq0} are, essentially, mutually exclusive.

One arrives at the problem of unification of the two solution theories of equation \eqref{eq0}, for instance, to treat  $b=b^{(1)}+b^{(2)}$, where $b^{(1)}$ is form-bounded and $b^{(2)}$ is from the Kato class. The two classes can be unified:
by the Heinz inequality and the interpolation, we have
$$
b^{(1)} \in \mathbf{F}_{\delta_1^2}, \quad b^{(2)} \in \mathbf{K}^{d+1}_{\delta_2^2} \quad \Rightarrow \quad b^{(1)}+b^{(2)} \in \mathbf{F}_\delta^{\scriptscriptstyle 1/2}, 
$$
where $\delta=\delta_1+\delta_2$ (we used the fact that 
$b \in \mathbf{F}^{\scriptsize 1/2}_\delta$ is equivalent to
$
\||b|^{\frac{1}{2}}(\lambda-\Delta)^{-\frac{1}{4}}\|_{2 \rightarrow 2} \leq \sqrt{\delta}
$
). 
However, as Theorem \ref{thm0}(\textit{ii}) shows, one should not be looking for the unification in the scale of $L^p$ solution spaces.
\end{remark}

\begin{example}Speaking of elementary examples of $\hat{b}$ in assertion (\textit{iii}), we single out the following class of time-homogeneous vector fields:
$$
\mathbf{M}'_\nu:=\big\{\hat{b}:\mathbb R^d \rightarrow \mathbb R^d \mid \hat{b}=\nabla(-\Delta)^{-1}W \text{ for some } W \in \mathbf{K}^d_\nu \big\},
$$
where (time-homogeneous) $W \in \mathbf{K}^d_\nu$ if and only if $\|(-\Delta)^{-1}|W|\|_\infty \leq \nu$.

To see that any $\hat{b}$ from $\mathbf{M}'_\nu$ satisfies the hypothesis of Theorem \ref{thm0__}, note that, for a given $\hat{b} \in \mathbf{M}'_\nu$, one automatically has ${\rm div\,}\hat{b} \in \mathbf{K}^d_\nu$. It is also clear that $\mathbf{M}'_\nu \subset \mathbf{K}^{d+1}_\nu$. 

In fact, any $\hat{b}$ from $\mathbf{M}_\nu'$ is form-bounded. Indeed, 
\begin{align*}
\langle \hat{b}^2|\psi|^2\rangle & =\langle(-\Delta)^{-1}W,W|\psi|^2\rangle-2 \langle(-\Delta)^{-1}W,\hat{b}|\psi||\nabla|\psi |\rangle \qquad \psi \in C_c^\infty(\mathbb{R}^d) \\
& \leq \nu\langle |W||\psi|^2\rangle+2\nu\langle \hat{b}^2|\psi|^2\rangle^\frac{1}{2}\|\nabla|\psi|\|_2 \\
& \leq \nu\langle |W||\psi|^2\rangle + \frac{1}{2}\langle \hat{b}^2|\psi|^2\rangle + 2\nu^2\|\nabla|\psi|\|_2^2.
\end{align*}
 Thus $\langle \hat{b}^2|\psi|^2\rangle\leq 2\nu \langle |W||\psi|^2\rangle+4\nu^2\|\nabla \psi\|_2^2$. It remains to note that $\mathbf{K}^d_\nu\subset \mathbf{F}_\nu$.
\end{example}

\begin{example}
One can modify the previous example by considering $\hat{b}=\hat{b}^{(1)}+\hat{b}^{(2)}$, where $\hat{b}^{(1)} \in \mathbf{M}'_\nu$ and
$$
\hat{b}^{(2)}:=(\phi_1(x)|x_2|^{-1+\varepsilon},\phi_2(x)|x_1|^{-1+\varepsilon},0,\dots,0), \quad \varepsilon \in ]0,1],
$$
where $\phi_1,\phi_2 \in C^\infty_c(\mathbb{R}^d)$. Clearly, $\hat{b}^{(2)}\in \mathbf{K}^{d+1}_\delta$, has divergence in $\mathbf{K}^{d}_\nu$, but is not in $L^2_{\loc}$ and hence is not form-bounded. This vector field satisfies the assumptions of assertion (d) in Theorem \ref{thm0__}.
\end{example}

\begin{remark}
In Theorem \ref{thm0__} we could consider
$$
b=\hat{b}+\tilde{b}+\check{b},
$$
where $\hat{b}$ and $\tilde{b}$ are as above, and $\check{b}=\nabla P$, $P \in [L^\infty(\mathbb R_+ \times \mathbb R^d)]^{d \times d}$, with ${\rm div\,}\check{b}$ form-bounded and in the Kato class. The corresponding analysis of Nash's $G$-functions follows the argument below and requires few modifications, see Remark \ref{P_remark}. Theorem \ref{thm0_} does not require a modification since, as is easily seen, $\check{b} \in \mathbf{M}_\delta$ with $\delta=\|P\|_\infty$.

\end{remark}

\begin{remark}
\label{rem_unif}
 The problem of unification of the class of form-bounded vector fields and the Kato class was addressed earlier in the simpler case of time-homogeneous vector fields $b=b(x)$ in \cite{S,Ki}.

First, it was noticed in \cite{S} that neither the form-boundedness nor the Kato class condition for $b=b(x)$ are responsible for the $(L^p,L^q)$ bound 
\begin{equation}
\label{LpLq}
\|u(t)\|_q \leq c_T t^{-\frac{d}{2}(\frac{1}{p}-\frac{1}{q})}\|f\|_p, \quad f \in L^p \cap L^q, \quad \frac{2}{2-\sigma^{-1}\sqrt{\delta}}< p<q \leq \infty.
\end{equation}
for the semigroup of $-\Delta + b \cdot \nabla$. In fact, it suffices to require that $b \in \mathbf{F}_\delta^{\scriptscriptstyle 1/2}$, i.e.\,$\||b|^{\frac{1}{2}}(\lambda-\Delta)^{-\frac{1}{4}}\|_{2 \rightarrow 2} \leq \sqrt{\delta}$.
It turned out that this new class of \textit{weakly form-bounded} vector fields contains the sums of vector fields $\mathbf{F}_{\delta^2}$ and $\mathbf{K}^{d+1}_{\delta^2}$.
\cite{S} also proposed a way to construct the semigroup generated by $-\Delta + b \cdot \nabla$, $b \in \mathbf{F}_\delta^{\scriptscriptstyle 1/2}$, $\delta<1$ in $L^2$ by ``guessing'' the resolvent of an appropriate operator realization $\Lambda_2$ of $-\Delta + b \cdot \nabla$:
\begin{equation}
\label{op_2}
\big(\zeta+\Lambda_2\big)^{-1}:=(\zeta-\Delta)^{-\frac{3}{4}}\bigl(1+ S \big)^{-1}(\zeta-\Delta)^{-\frac{1}{4}},
\end{equation}
where $\Real \zeta \geq \frac{d}{d-1}\lambda$, the operator $S(\zeta):=(\zeta-\Delta)^{-\frac{1}{4}}b \cdot \nabla (\zeta-\Delta)^{-\frac{3}{4}}$ is bounded in $L^2$ by $b \in \mathbf{F}_\delta^{\scriptscriptstyle 1/2}$. The RHS of \eqref{op_2} coincides, after expanding $(1+ S)^{-1}$ in the geometric series, with the Neumann series for $-\Delta + b \cdot \nabla$.

Next, it was shown in \cite{Ki} that equation \eqref{eq0} with $b \in \mathbf{F}_\delta^{\scriptscriptstyle 1/2}$ has a detailed  $L^p$ regularity theory for $p \in I_\delta$, where the open interval $I_\delta$, centered around $2$, expands to $]1,\infty[$ as $\delta \downarrow 0$. However one has to guess the resolvent differently:
\begin{equation}
\label{op_p}
\big(\zeta+\Lambda_p\big)^{-1}:=(\zeta-\Delta)^{-1}-(\zeta-\Delta)^{-\frac{1}{2}-\frac{1}{2q}}Q_p(q)\big(1+T_p\big)^{-1}R_p(r)(\zeta-\Delta)^{-\frac{1}{2r'}}, 
\end{equation}
where $1 \leq r<p<q<\infty$, $T_p(\zeta):=b^{\frac{1}{p}}\cdot \nabla (\zeta-\Delta)^{-1}|b|^{\frac{1}{p'}} \in \mathcal B(L^p)$ and $Q_p(q)$, $R_p(r) \in \mathcal B(L^p)$ are such that one obtains again, after expanding $(1+T_p)^{-1}$ in the geometric series, the Neumann series for $-\Delta + b \cdot \nabla$. (Note that the direct analogue of \eqref{op_2} in $L^p$ requires a much more restrictive condition: $|b|$ is in the weak $L^d$ space.) From \eqref{op_p}, one obtains right away $L^p$ regularity of the $1+\frac{1}{q}$-order spatial derivatives of solutions to the corresponding parabolic equation, hence the corresponding Feller semigroup. See also further developments in \cite{KiS_theory}. The Feller semigroup determines, for every starting point, a weak solution to the corresponding SDE \cite{KiS_BM}. The crucial point here is that one works in $L^p$ while keeping intact the $L^2 \rightarrow L^2$ assumption on the vector field $b$ (i.e.\,weak form-boundedness) and hence the class of its admissible singularities, except for requiring a smaller $\delta$. 

The proof that the operator-valued function in the RHS of \eqref{op_2} determines the resolvent of the generator of a $C_0$ semigroup is delicate. For \eqref{op_p} the situation is more difficult. In both proofs, one needs the Trotter Approximation Theorem and Hille's theory of pseudo-resolvents. Both proofs depend crucially on the holomorphy of the constructed semigroup since the latter are only quasi-bounded.

This paper concerns with the time-inhomogeneous case $b=b(t,x)$, which presents the next level of difficulty.
Having at hand the evolution family for \eqref{eq0} in $\mathcal W^{\frac{1}{2},2}$, provided by Theorem \ref{thm0}(\textit{ii}), we now approach the problem of constructing the evolution family in $L^2$ as a separate problem. The construction in \cite{S,Ki} of the semigroup directly in $L^2$ and in $L^p$ can be viewed now as solving
several problems at the same time. We expect that the spatial $\mathcal W^{1+1/q,p}$ regularity of the evolution family constructed in Theorem \ref{thm0}, for $q>p$ and $p$ large, can be obtained with additional effort, as well as the ensuing Feller property and a weak well-posedness of the corresponding SDE.  
\end{remark}

\appendix

\section{}

\label{appg}

Below we obtain the energy inequality \eqref{ei2} on a fixed interval $0 \leq s<t \leq T$ assuming that $b$ satisfies \eqref{weak_fbd_g} with $g \in L^2_{\loc}(\mathbb R_+)$.
We will be working at the a priori level, i.e.\,with $b$, $f$ additionally assumed to be bounded and smooth, and $u(s)=f$.

We use notations introduced in the beginning of the proof of Theorem \ref{thm0}.

Multiplying \eqref{eq0} by $(\lambda-\Delta)^{\frac{1}{2}}u$ and integrating, we obtain 
$$
\|u(t)\|_{H}^2 + 2\int_s^t \|u(r)\|_{H_+}^2 dr + 2\int_s^r \langle b(r) \cdot \nabla u,(\lambda-\Delta)^{\frac{1}{2}}u\rangle dr \leq \|f\|_{H}^2.
$$
The term to control:
\begin{align*}
\big|\langle b(t) \cdot \nabla u,(\lambda-\Delta)^{\frac{1}{2}}u\rangle \big| & \leq \big\||b(t)|^{\frac{1}{2}}|\nabla u|\big\|_2\big\||b(t)|^{\frac{1}{2}}|(\lambda-\Delta)^{\frac{1}{2}}u|\big\|_2.
\end{align*}
By \eqref{weak_fbd_g},
\begin{align*}
\||b(t)|^{\frac{1}{2}}|(\lambda-\Delta)^{\frac{1}{2}}u|\big\|^2_2 & \leq \delta \|(\lambda-\Delta)^{\frac{1}{4}}|(\lambda-\Delta)^{\frac{1}{2}}u|\|_2^2 + g(t)\|u\|^2_{\mathcal W^{1,2}} \\
& (\text{we are using the Beurling-Deny inequality}) \\
& \leq \delta \|u\|^2_{H_+} + g(t)\|u\|^2_{\mathcal W^{1,2}}.
\end{align*}
Similarly, a variant of the Beurling-Deny inequality: $\|(\lambda-\Delta)^{\frac{1}{4}}|\nabla u|\|_2 \leq \|(\lambda-\Delta)^{\frac{1}{4}}\nabla u\|_2$ and integration by parts yield
$$\big\||b(t)|^{\frac{1}{2}}|\nabla u|\big\|_2^2 \leq \delta \|u\|^2_{H_+} + g(t)\|u\|^2_{\mathcal W^{1,2}}.$$
Finally, we estimate $g(t)\|u\|^2_{\mathcal W^{1,2}} \leq g(t)\|u\|_{H_+}\|u\|_{H}$, so
$$
\int_s^t g(r)\|u\|^2_{\mathcal W^{1,2}} dr \leq \varepsilon \int_s^t \|u(r)\|^2_{H_+} dr  + \frac{1}{4\varepsilon}\int_s^t g^2(r)\|u\|_{H}^2 dr.
$$
Hence 
$$
\|u(t)\|_{H}^2 + 2(1-\delta-\varepsilon)\int_s^t \|u(r)\|_{H_+}^2 dr - \frac{1}{4\varepsilon}\int_s^t g^2(r)\|u\|_{H}^2 dr \leq \|f\|_{H}^2,
$$
where $\varepsilon$ is sufficiently small so that $1-\delta-\varepsilon>0$. It follows that
$$
\sup_{r \in [0,t]}\|u(r)\|_{H}^2 + 2(1-\delta-\varepsilon)\int_s^t \|u(r)\|_{H_+}^2 dr \leq \frac{1}{4\varepsilon}\int_s^t g^2(r)\|u\|_{H}^2 dr + \|f\|_{H}^2,
$$
so, assuming first that $s$, $t$ are sufficiently close so that $\frac{1}{4\varepsilon}\int_s^t g^2(r)dr<1$, and then using the reproduction property, we obtain
$$
C_1\sup_{r \in [0,t]}\|u(r)\|_{H}^2 + C_2\int_s^t \|u(r)\|_{H_+}^2 dr\leq \|f\|_{H}^2 
$$
for appropriate $C_1$, $C_2>0$, as required.

\bigskip

\section{Proof of Proposition \ref{prop1}}

\label{lions_app}

1. Fix some $[s_1,t_1] \subset ]s,T[$. Let us show that $\partial_t u \in L^2([s_1,t_1],H_-)$. Put
$$
F_t(\varphi):=\langle (\lambda-\Delta)^{\frac{3}{4}} u, (\lambda-\Delta)^{\frac{3}{4}}\varphi \rangle + \langle b(t)\cdot \nabla u,(\lambda-\Delta)^{\frac{1}{2}}\varphi\rangle \qquad \varphi \in L^2_{com}(]s,T[,H_+).
$$
By Proposition \ref{lem_fbd_est}, 
$$
|F_t(\varphi)| \leq (1+\delta) \|u(t)\|_{\mathcal W^{\frac{3}{2}{2}}}\|\varphi(t)\|_{\mathcal W^{\frac{3}{2}{2}}}.
$$
Hence by the Riesz Representation Theorem there exists a unique $w(t) \in H_+$ such that
$$
F_t(\varphi))=\langle (\lambda-\Delta)^{\frac{3}{4}}w(t),(\lambda-\Delta)^{\frac{3}{4}}\varphi(t) \rangle,
$$
where $\|w(t)\|_{H_+} \leq (1+\delta)\|u(t)\|_{H_+}$, so $w \in L^2([s_1,t_1],H_+)$.

 In terms of $w(t)$, the hypothesis that $u$ is a weak solution becomes
$$
-\int_s^T \langle u(t),\partial_t \varphi \rangle_{H} dt + \int_s^T \langle w(t),\varphi(t) \rangle_{H_+} dt=0. 
$$
In particular, taking $\varphi=\psi \eta$, where $\psi \in H_+$, $\eta \in C_c^\infty(]s_1,t_1[)$, we have
$$
-\int_{s_1}^{t_1} \langle u(t),\psi \rangle_{H} \eta' dt + \int_{s_1}^{t_1} \langle w(t),\psi \rangle_{H_+} \eta dt=0.
$$
Since 
$$
\big| \int_{s_1}^{t_1} \langle w(t),\psi \rangle_{H_+} \eta dt \big| \leq \|\psi\|_{H_+}\|w\|_{L^2([s_1,t_1],H_+)}\|\eta\|_{L^2[s_1,t_1]},$$
we have
$$
\big|\int_{s_1}^{t_1} \langle u(t),\psi \rangle_{H} \eta' dt \big| \leq  (1+\delta)\|\psi\|_{H_+}\|u\|_{L^2([s_1,t_1],H_+)} \|\eta\|_{L^2[s_1,t_1]}.
$$
It follows that
$
\frac{d}{dt}\langle u(t),\psi \rangle_{H} \in L^2[s_1,t_1]$, $\psi \in H_+
$
and, furthermore,
$$
\big\| \frac{d}{dt}\langle u(t),\psi \rangle_{H} \big\|_{L^2[s_1,t_1]} \leq  (1+\delta)\|\psi\|_{H_+}\|u\|_{L^2([s_1,t_1],H_+)}. 
$$
Hence there exists $\partial_t u \in L^2([s_1,t_1],H_-)$ such that
$$
-\int_s^T \langle \partial_t u(t),\psi \rangle_{H_-,H_+}  \eta dt =  \int_s^T \langle u(t),\psi\rangle_{H} \eta ' dt,
$$
where $\langle \cdot,\cdot \rangle_{H_-,H_+}$ denotes the $H_+$, $H_-$ pairing,
so
\begin{align*}
-\int_s^T \langle \partial_t u(t),\psi \eta \rangle_{H_-,H_+} dt & =  \int_s^T \langle u(t),(\psi \eta)'\rangle_{H}  dt \\
& =  \int_s^T \big(\langle (\lambda-\Delta)^{\frac{3}{4}} u, (\lambda-\Delta)^{\frac{3}{4}}\psi \eta \rangle + \langle b(t)\cdot \nabla u,(\lambda-\Delta)^{\frac{1}{2}}\psi \eta \rangle\big) dt.
\end{align*}
Since $\{\psi \eta \mid \psi \in H_+, \eta \in C_c^\infty(]s_1,t_1[)\}$ is dense in $L^2([s_1,t_1],H_+)$, we have
\begin{align*}
-\int_s^T \langle \partial_t u(t),\varphi \rangle_{H_-,H_+} dt  =  \int_s^T \big(\langle (\lambda-\Delta)^{\frac{3}{4}} u, (\lambda-\Delta)^{\frac{3}{4}}\varphi \rangle + \langle b(t)\cdot \nabla u,(\lambda-\Delta)^{\frac{1}{2}}\varphi \rangle\big) dt
\end{align*}
for all $\varphi \in L^2([s_1,t_1],H_+)$. 

In particular, taking $\varphi=u$ and using $b \in L^\infty\,\mathbf F^{\scriptscriptstyle 1/2}$ as in the proof of Proposition \ref{lem_fbd_est}, we have
\begin{equation}
\label{ei}
\int_{s_1}^{t_1} \langle \partial_t u(t),u \rangle_{H_-,H_+} dt  + (1-\delta)\int_{s_1}^{t_1} \|(\lambda-\Delta)^{\frac{3}{4}} u\|_2^2 dt \leq 0.
\end{equation}

2. We  put $u_\varepsilon:=\xi_\varepsilon \ast u$, where $\xi_\varepsilon$ is a Friedrichs mollifier of compact support. Since $H_-$, $H_+$ are separable spaces, we have
$u_\varepsilon(t) \rightarrow u(t)$ in $H_+$ for a.e.\,$t \in ]s,T[$, $u_\varepsilon \rightarrow u$ in $L^2([s_1,t_1],H_+]$, $u'_\varepsilon \rightarrow u'$ in $L^2([s_1,t_1],H_-]$

Further, we have
\begin{align*}
\frac{d}{dt}\|u_\varepsilon(t)-u_\delta(t)\|_{H} & = \frac{d}{dt}\|(\lambda-\Delta)^{\frac{1}{4}}u_\varepsilon(t)-(\lambda-\Delta)^{\frac{1}{4}}u_\delta(t)\|_{2} \\
& = 2 \langle (\lambda-\Delta)^{\frac{1}{4}}\partial_t u_\varepsilon(t) - (\lambda-\Delta)^{\frac{1}{4}}\partial_t u_\delta(t), (\lambda-\Delta)^{\frac{1}{4}}u_\varepsilon(t) - (\lambda-\Delta)^{\frac{1}{4}} u_\delta(t)\rangle
\end{align*}
so
\begin{align*}
&\|u_\varepsilon(t_1)-u_\delta(t_1)\|_{H} - \|u_\varepsilon(s_1)-u_\delta(s_1)\|_{H} \\
& = 2 \int_{s_1}^{t_1} \langle (\lambda-\Delta)^{-\frac{1}{4}}\partial_r u_\varepsilon - (\lambda-\Delta)^{-\frac{1}{4}}\partial_r u_\delta, (\lambda-\Delta)^{\frac{3}{4}}u_\varepsilon - (\lambda-\Delta)^{\frac{3}{4}} u_\delta\rangle dr.
\end{align*}
Hence, fixing $s_1 \in ]s,T[$ such that  $u_\varepsilon(s_1) \rightarrow u(s_1)$ in $H_+$, we obtain
$$
\limsup_{\varepsilon, \delta \downarrow 0}\sup_{t \in [s_1,t_1]}\|u_\varepsilon(t)-u_\delta(t)\|_{H} \leq \int_0^1 \|\partial_r u_\varepsilon - \partial_r u_\delta\|_{H_-}^2 dr + \int_{s_1}^{t_1} \|u_\varepsilon -  u_\delta\|_{H_+}^2 dr \rightarrow 0
$$
as $\varepsilon,\delta \downarrow 0$. It follows that $\{u_\varepsilon\}$ converges in $L^\infty([s_1,t_1],H)$ to $u$, and so  $u \in C([s_1,t_1],H)$. This gives (\textit{ii}).

3. Finally, we note that $\frac{d}{dt}\|u(t)\|^2_{H}=2\langle u'(t),u(t)\rangle_{H}$, as follows from
$$
\|u_\varepsilon(t_1)\|^2_{H} - \|u_\varepsilon(s_1)\|^2_{H}=2\int_{s_1}^{t_1} \langle  (u_\varepsilon)',u_\varepsilon\rangle_{H} dr
$$
upon taking the limit $\varepsilon \downarrow 0$. Combining this with \eqref{ei} we obtain (\textit{iii}).

\bigskip

\section{}

\label{examples_sect}

1.~Let us first show that if $
b \in L^\infty(\mathbb R_+,L^d)$,
then for a.e.\,$t \in \mathbb R_+$
\begin{equation}
\label{fbb_2}
\||b(t)|\psi\|^2_2 \leq \delta \|(-\Delta)^{\frac{1}{2}} \psi\|^2_2, \quad \psi \in W^{1,2},
\end{equation}
for $\delta:=\sup_{t \in \mathbb R_+} \|b(t)\|_d^2<\infty$. Indeed,
\begin{align*}
\|b(t)\psi\|_2^2 
& \leq \|b(t)\|_d^2 \|\psi\|_{\frac{2d}{d-2}}^2\\
& (\text{we are applying the Sobolev Embedding Theorem}) \\
& \leq C_S\|b(t)\|_d^2 \|\nabla \psi\|_2^2,
\end{align*}
as claimed.
Now,  applying Heinz' inequality in \eqref{fbb_2}, we obtain $b \in L^\infty\mathbf{F}_{\sqrt{\delta}}^{\scriptscriptstyle 1/2}$, so $b \in \mathbf{MF}_{\sqrt{\delta}}$ by \eqref{incl}.

\medskip

If $b$ belongs to the critical Ladyzhenskaya-Prodi-Serrin class \eqref{LPS_c}, i.e.
$$
b \in L^p([0,\infty[,L^q), \quad \frac{d}{q}+\frac{2}{p} \leq 1, \quad p \geq 2, \quad q \geq d
$$
then we estimate
\begin{align*}
|b(t,x)|  =\frac{|b(t,x)|}{\langle |b(t,\cdot)|^q\rangle^{\frac{1}{q}}}\langle |b(t,\cdot)|^q\rangle^{\frac{1}{q}} \leq \frac{d}{q} \biggl(\frac{|b(t,x)|^q}{\langle |b(t,\cdot)|^q\rangle} \biggr)^{\frac{1}{d}} + \frac{2}{p}\bigl( \langle |b(t,\cdot)|^q\rangle^{\frac{1}{q}} \bigr)^{\frac{p}{2}},
\end{align*}
where the first term is in $ L^\infty([0,\infty[,L^d)$ and the second term is in $L^2([0,\infty[,L^\infty)$. In view of the previous example, it is clear that $b \in \mathbf{MF}_\delta$.

\medskip

2.~One has
$$
\sup_{t \in \mathbb R}\|b(t)\|_{M_{1+\varepsilon}} <\infty \text{ (see \eqref{morrey})} \quad \Rightarrow \quad b \in  L^\infty\mathbf{F}_{\delta}^{\scriptscriptstyle 1/2}
$$
with $\delta$ proportional to the Morrey norm $\sup_{t \in \mathbb R}\|b(t)\|_{M_{1+\varepsilon}}$,
see \cite[Theorem 7.3]{A}. So, in particular, if $\sup_{t \in \mathbb R}\|b(t)\|_{M_{1+\varepsilon}} <\infty$, then $b \in \mathbf{MF}_\delta$.


\bigskip

\begin{thebibliography}{99}


\bibitem[A]{A} D. Adams, Weighted nonlinear potential theory, {\em Trans. Amer. Math. Soc.}
\textbf{297} (1986), 73-94.

\bibitem[Ar]{Ar} D.G.\,Aronson, ``Non-negative solutions of linear parabolic equations'', \textit{Ann.\,Sc.\,Norm.\,Sup.\,Pisa (3)} \textbf{22} (1968), 607-694.



\bibitem[BC]{BC} R.~Bass and Z.-Q.~Chen, \newblock Brownian motion with singular drift.
\newblock {\em Ann.~Prob.}, 31 (2003), 791-817.

\bibitem[CL]{CL} E.\,Carlen and M.\,Loss, Optimal smoothing and decay estimates for viscously damped conservation laws, with applications to the 2-D Navier-Stokes equations, {\em Duke Math. J.}, \textbf{81}(1) (1995), 135-157.


\bibitem[CLMS]{CLMS} R. Coifman, P.-L. Lions, Y. Meyer and S. Semmes, Compensated compactness and Hardy spaces,
{\em J. Math. Pures Appl.} \textbf{72} (1992) 247-286.


\bibitem[D]{Da} E.B.\,Davies, Heat Kernels and Spectral Theory, Cambridge University Press, 1989.

\bibitem[DG]{DG} E.\,De Giorgi, Sulla differenziabilit\`{a} e l'analiticit\`{a} delle estremali degli integrali multipli regolari, \textit{Mem.\,Acc.\,Sci.\,Torino} 3 (1957), 25-43.


\bibitem[FS]{FS} E.\,B.\,Fabes and D.\,W.\,Stroock, A new proof of Moser's parabolic Harnack
inequality via the old ideas of Nash, \textit{Arch. Ratl. Mech. and Anal.} \textbf{96} (1986), 327-338.

\bibitem[K]{Ki} D.\,Kinzebulatov, A new approach to the $L^p$-theory of $-\Delta + b\cdot\nabla$, and its applications to Feller processes with general drifts,
\newblock {\em Ann.~Sc.~Norm.~Sup.~Pisa (5)}, \textbf{17} (2017), 685-711. 


\bibitem[K2]{Ki2} D.\,Kinzebulatov, Parabolic equations and SDEs with time-inhomogeneous Morrey drift, arXiv:2301.13805.

\bibitem[KS1]{KiS_MathAnn} D.\,Kinzebulatov and Yu.\,A.\,Sem\"{e}nov, Heat kernel bounds for parabolic equations with singular (form-bounded) vector fields, {\em Math.\,Ann.}, \textbf{384} (2022), 1883-1929.



\bibitem[KS2]{KiS_theory} D.\,Kinzebulatov and Yu.\,A.\,Sem\"{e}nov, On the theory of the Kolmogorov operator in the spaces $L^p$ and $C_\infty$, {\em Ann. Sc. Norm. Sup. Pisa (5)} \textbf{21} (2020), 1573-1647.


\bibitem[KS3]{KiS_BM} D.\,Kinzebulatov and Yu.A.\,Sem\"{e}nov, Brownian motion with general drift, \newblock{\em Stoc. Proc. Appl.}, \textbf{130} (2020), 2737-2750





\bibitem[KS4]{KiS_nash} D.\,Kinzebulatov and Yu.\,A.\,Sem\"{e}nov, Kolmogorov operator with the vector field in Nash class, {\em Tohoku Math J.}, \textbf{74}(4) (2022), 569-596



\bibitem[LM]{LM} J.\,-L.\,Lions and E.\,Magenes, Non-Homogeneous
Boundary Value Problems and Applications, {\em Springer}, 1972.


\bibitem[LZ]{LZ} V.\,Liskevich and Q.\,S.\,Zhang, Extra regularity for parabolic equations with drift
terms, \textit{Manuscripta Math.} \textbf{113} (2004), 191-209.


\bibitem[MV]{MV}
V.\,G.\,Mazya and I.\,E.\,Verbitsky, Form boundedness of the general
second-order differential operator, {\em Comm.\,Pure Appl.\,Math.} \textbf{59} (2006), 1286-1329. 




\bibitem[N]{N} J.\,Nash, Continuity of solutions of parabolic and elliptic equations, \newblock \textit{Amer.\,Math.\,J.} \textbf{80} (1) (1958), 931-954.

\bibitem[O]{O} H.\,Osada, Diffusion processes with generators of generalized divergence form, \textit{J. Math. Kyoto Univ.,}
\textbf{27}(4) (1987), 597-619.


\bibitem[QX1]{QX} Z.\,Qian, G.\,Xi, Parabolic equations with singular divergence-free drift vector fields, {\em J. London Math. Soc.} \textbf{100} (1) (2019), 17-40.

\bibitem[QX2]{QX2} Z.\,Qian, G.\,Xi, Parabolic equations with divergence-free drift in space $L_t^lL_x^q$, {\em Indiana Univ.\,Math.\,J.} \textbf{68}(3) (2019), 761-797.



\bibitem[S]{S} Yu.\,A.\,Sem\"{e}nov, \newblock Regularity theorems for parabolic equations, \newblock {\em J.\,Funct.\,Anal.}, \textbf{231} (2006), 375-417.

\bibitem[SSSZ]{SSSZ} G.\,Seregin, L.\,Silvestre, V.\,\v{S}verak and A.\,Zlato\v{s}, On divergence-free drifts, {\em J. Differential Equations}, \textbf{252}(1) (2012), 505-540.




\bibitem[Z1]{Z1} Q.\,S.\,Zhang, A strong regularity result for parabolic equations, {\em Comm.\,Math.\,Phys.} \textbf{244} (2004) 245-260.



\bibitem[Z2]{Z2} Q.\,S.\,Zhang, Gaussian bounds for the fundamental solutions
of $\nabla (A\nabla u) + B\nabla u - u_t = 0$, {\em Manuscripta Math.} \textbf{91} (1997), 381-390.


\bibitem[Z3]{Z} Q.\,S.\,Zhang, On a parabolic equation with a singular lower order term, {\em Trans.\,Amer.\,Math.\,Soc.} \textbf{348} (1996), 2811-2844.

\bibitem[Zh]{Zh} V. V. Zhikov, Remarks on the uniqueness of a solution
of the Dirichlet problem for second-order elliptic
equations with lower-order terms, {\em Funktsional. Anal.
i Prilozhen.} (2004) \textbf{38}, 15-28.


\end{thebibliography}
\end{document}